\documentclass[reqno,12pt,letterpaper]{amsart}
\usepackage{amsmath,amssymb,amsthm,graphicx,mathrsfs,url}
\usepackage[usenames,dvipsnames]{color}
\usepackage[colorlinks=true,linkcolor=Red,citecolor=Green]{hyperref}
\usepackage{amsxtra}

\def\arXiv#1{\href{http://arxiv.org/abs/#1}{arXiv:#1}}

\def\?[#1]{\textbf{[#1]}\marginpar{\Large{\textbf{??}}}}

\def\smallsection#1{\smallskip\noindent\textbf{#1}.}
\let\epsilon=\varepsilon 

\newcommand{\RR}{{\mathbb R}}
\newcommand{\NN}{{\mathbb N}}
\newcommand{\CC}{{\mathbb C}}

\newtheorem{prop}{Proposition}[section]

\newtheorem{lemm}[prop]{Lemma}

\numberwithin{equation}{section}

\def\bbbone{{\mathchoice {1\mskip-4mu {\rm{l}}} {1\mskip-4mu {\rm{l}}}
{ 1\mskip-4.5mu {\rm{l}}} { 1\mskip-5mu {\rm{l}}}}}

\DeclareMathOperator{\Res}{Res}

\DeclareMathOperator{\comp}{comp}

\DeclareMathOperator{\Ell}{ell}
\DeclareMathOperator{\Hom}{Hom}

\let\Im=\Imag
\DeclareMathOperator{\loc}{loc}

\let\Re=\Real

\DeclareMathOperator{\supp}{supp}

\DeclareMathOperator{\WF}{WF}
\DeclareMathOperator{\tr}{tr}

\def\WFh{\WF_h}

\title[Dynamical zeta functions via microlocal analysis]%
{Dynamical zeta functions for Anosov flows\\
via microlocal analysis}
\author{Semyon Dyatlov}
\email{dyatlov@math.mit.edu}
\address{Department of Mathematics, Massachusetts Institute of Technology,
77 Massachusetts Ave, Cambridge, MA 02139}
\author{Maciej Zworski}
\email{zworski@math.berkeley.edu}
\address{Department of Mathematics, University of California,
Berkeley, CA 94720, USA}

\begin{document}

\begin{abstract}
The purpose of this paper is to give a short microlocal proof
of the meromorphic continuation of the Ruelle zeta function for
$C^\infty$  Anosov flows. More general results have
been recently proved by Giulietti--Liverani--Pollicott \cite{glp}
but our approach is different and is based on the study of the generator of the flow
as a semiclassical differential operator.

\end{abstract}

\maketitle

\addtocounter{section}{1}
\addcontentsline{toc}{section}{1. Introduction}

The purpose of this article is to provide a short 
microlocal proof of the meromorphic continuation of the Ruelle zeta function
for $ C^\infty $ Anosov flows on compact manifolds:

\medskip
\noindent
{\bf Theorem.} {\em 
Suppose $ X$ is a compact manifold and $ \varphi_t : X \to X $
is a $C^\infty$ Anosov flow with orientable stable and unstable bundles. Let $ \{ \gamma^\sharp \}  $ denote
the set of primitive orbits of $ \varphi_t $, with $ T^\sharp_\gamma $
their periods. Then the Ruelle zeta function, 
\[  \zeta_{\rm{R}} ( \lambda ) = \prod_{\gamma^\sharp } ( 1 - e^{ i \lambda T_\gamma^\sharp
} ) , \]
which converges for $ \Im \lambda \gg 1 $ has a meromorphic
continuation to $ \mathbb C $.}

\medskip

In fact the proof applies to any Anosov flow for which linearized Poincar\'e
maps $ \mathcal P_\gamma $ for
closed orbits $ \gamma $ satisfy
\begin{equation}
\label{eq:Poinc}
| \det ( I - \mathcal P_\gamma ) | = ( -1)^{q} \det ( I - \mathcal P_\gamma ) , \ \text{
  with $ q $ independent of $ \gamma$.}
\end{equation}
A class of examples is provided by $ X = S^*M $ where $ M $ is a compact orientable negatively curved manifold
with $ \varphi_t $ the geodesic flow -- see \cite[Lemma B.1]{glp}.
For methods which can be used to eliminate the orientability
assumptions see \cite[Appendix B]{glp}.

The meromorphic continuation of $ \zeta_{\rm{R}} $  was conjectured by Smale \cite{Sm} and in greater 
generality it was proved very recently by Giulietti, Liverani, and
Pollicott \cite{glp}.
Another recent perspective on dynamical zeta functions in the contact case has been provided
by Faure and Tsujii~\cite{fa-ts,fa-ts2}. 
Our motivation and proof are however different from those of~\cite{glp}: we were
investigating trace formul{\ae} for Pollicott--Ruelle resonances \cite{Po,Ru2} which 
give some lower bounds on their counting function. Sharp upper
bounds were given recently in \cite{ddz,fa-sj}.

To explain the trace formula for resonances suppose first that $ X = S^* \Gamma \backslash {\mathbb H}^2 $ is a compact
Riemann surface. Then the Selberg trace formula combined with the
Guillemin trace formula \cite{Gu} gives 
\begin{equation}
\label{eq:poison1} \sum_{  \mu \in \Res ( P )  } e^{ - i \mu  t } = \sum_{
    \gamma} 
\frac{ T_\gamma^\#     \delta (
  t - T_\gamma ) } { | \det ( I - \mathcal P_\gamma ) 
  |} , \ \  t > 0 , \end{equation}
see~\cite{Leb} for an accessible presentation in the physics literature
and~\cite{DGF} for the case of higher dimensions. 
On the left hand side $ \Res ( P ) $ is the set of resonances of $ P =
- i V $ where $ V $ is the generator of the flow,
\[ \Res ( P ) = \left\{   \mu_{j,k}  = \lambda_j  - i ( k +
\textstyle{\frac12} ) ,\  j,
k \in \mathbb N \right\} , \]
where $ \lambda_j $'s are the zeros of the Selberg zeta function
included according to their multiplicities. 
On the right hand side 
$\gamma $'s  are periodic orbits,  $ \mathcal P_\gamma $ is the
linearized  Poincar\'e map, $ T_\gamma $ is the period of $ \gamma $, and $
T_\gamma^\# $ is the primitive period.  

The point of view of Faure--Sj\"ostrand \cite{fa-sj} stresses the 
analogy between analysis of the propagator $ \varphi_{-t}^* = e^{
  - i t P } $ with scattering theory for elliptic operators on
non-compact manifolds: for flows, 
the fiber infinity of $ T^* X $ is the analogue of spatial infinity 
for
scattering on non-compact manifolds. Melrose's Poisson formula for
resonances valid for Euclidean infinities \cite{mel1,sz,zw1}  and some
hyperbolic infinities \cite{GZ} suggests that \eqref{eq:poison1}
should be valid for general Anosov flows but that 
seems to be unknown. 

In general, the validity of \eqref{eq:poison1} follows
from the finite order (as an entire function) 
of the analytic continuation of 
\begin{equation}
\label{eq:zeta1} 
 \zeta_1 ( \lambda ) :=  \exp \bigg( - \sum_\gamma {T_\gamma^\# e^{ i
  \lambda T_\gamma} \over T_\gamma | \det ( I - \mathcal P_\gamma ) |\
 }\bigg).
\end{equation}
The $ \mu$'s appearing on the left hand side of \eqref{eq:poison1}
are the zeros of $ \zeta_1$ -- 
see \cite[\S 5]{GZ} or \cite{zw1} for an indication of this simple
fact.
Under certain analyticity assumptions on  $ X $ and $ \varphi_t $,
Rugh
\cite{Ru} and Fried \cite{Fr} showed that the Ruelle zeta function 
$ \zeta_R ( \lambda ) $ is a meromorphic function of finite
order but neither \cite{glp} nor our paper suggest the validity 
of such a statement in general.


One reason to be interested in \eqref{eq:poison1} in the general 
case is the 
following consequence based on \cite[\S4]{gz}:
the counting function for the
Pollicott--Ruelle resonances in wide strips cannot be sublinear. More
precisely, there exists a constant $ C_0 $ such that 
for each $ \epsilon \in( 0,1) $,
\begin{equation}
\label{eq:low} 
  \# \{ \mu \in \Res ( P ) \; : \; \Im \mu > - {{C_0}/{\epsilon}} ,  \ 
| \mu | \leq r \} \not < r^{1-\epsilon } , \ \ r \geq C(\epsilon) ,
\end{equation}
see \cite{JiZw} and comments below.

We arrived at the proof of main Theorem while attempting to demonstrate~\eqref{eq:poison1}
for $C^\infty$ Anosov flows.
We now indicate the idea of that proof in the case of 
analytic continuation of $ \zeta_1 ( \lambda ) $ given by 
\eqref{eq:zeta1}. It 
converges for $ \Im \lambda \gg 1 $ -- see Lemma~\ref{l:dyn-4} for convergence and~\eqref{eq:Rue}
below for the connection to the Ruelle zeta function. The starting point is Guillemin's formula,
\begin{equation}
\label{eq:ABG}  \tr^\flat e^{-it P }  = 
\sum_{
    \gamma} 
\frac{ T_\gamma^\#     \delta (
  t - T_\gamma ) } { | \det ( I - \mathcal P_\gamma ) |} , \ \ t > 0
\end{equation}
where the trace is defined using distributional operations of 
pullback  by $ \iota ( t, x ) = ( t, x, x ) $ and pushforward by 
$ \pi: ( t, x ) \to t $: $ \tr^\flat e^{-it P} := \pi_* \iota^* K_{e^{-itP} }, $ where $ K_\bullet $ denotes the distributional 
kernel of an operator. The pullback is well-defined in the sense of distributions~\cite[\S8.2]{ho1} because
the wave front set of $K_{e^{-itP}}$ satisfies
\begin{equation}  
\label{wfitP}
\WF ( K_{e^{ - i t P }} ) \cap N^*(\mathbb R_t\times\Delta(X)) =
\emptyset, \ \ \ t > 0 ,
\end{equation}
where $\Delta(X)\subset X\times X$ is the diagonal and
$N^*(\mathbb R_t\times \Delta(X))\subset T^*(\mathbb R_t\times X\times X)$
is the conormal bundle. See Appendix~\ref{s:guillemin}
and~\cite[\S II]{Gu} for details.

Since 
\[ \begin{split}  \frac d {d \lambda} \log  \zeta_1(\lambda) 
= {1\over i} \sum_{ \gamma }
\frac{ T_\gamma^\# e^{ i \lambda T_\gamma}}  { |\det( I - \mathcal P_\gamma) |} 
= {1\over i}\int_0^\infty e^{ i t \lambda }\tr^\flat e^{-i t P }  dt
, \end{split} \]
it is enough to show that the right hand side has a meromorphic
continuation to $ \CC $ with simple poles and residues which are
non-negative integers.  For that it is enough to take $t_0>0$ smaller than
$T_\gamma$ for all $\gamma$ (note that $\tr^\flat e^{-itP}=0$ on $(0,t_0)$) and
consider a continuation
of  
\[   {1\over i}\int_{t_0} ^\infty e^{it \lambda } \tr^\flat  e^{ - i t P }  dt = 
{1\over i}e^{ i t_0 \lambda } \int_{0} ^\infty e^{ i t \lambda }\tr^\flat \varphi_{-t_0} ^* e^{-i t P }
 dt .\]

We now note that 
\begin{equation}
  \label{e:phoenix}
i\int_{0}^\infty e^{ i t \lambda } \varphi_{-t_0}^*  e^{ - i t P } 
dt   = \varphi_{-t_0}^* ( P
- \lambda )^{-1}\quad\text{for }\Im \lambda \gg 1 .
\end{equation}
With a justification provided by a simple approximation argument
(see the proof of~\cite[Theorem~19.4.1]{ho3} for a similar construction)
it is then sufficient to continue
\begin{equation}
\label{eq:flattrr}  \tr^\flat  \left(  \varphi_{-t_0}^* ( P
- \lambda )^{-1} \right) ,  \ \ \ \Im \lambda \gg 1 , \end{equation}
meromorphically.  As recalled in \S\ref{stuff}, $ ( P - \lambda )^{-1}
: C^\infty ( X )  \to {\mathcal D' } ( X ) $ continues meromorphically
so to check the meromorphy of \eqref{eq:flattrr} we only need to
check the analogue of the wave front set relation \eqref{wfitP}
for the distributional kernel of $ \varphi_{-t_0}^* ( P - \lambda )^{-1} $,
namely that this wave front set does not intersect $N^*(\Delta(X))$. But that follows
from an adaptation of propagation results of Duistermaat--H\"ormander
\cite[\S 26.1]{ho3}, Melrose \cite{mel}, and Vasy \cite{vasy1}. 
The Faure--Sj\"ostrand spaces \cite{fa-sj} provide the a priori
regularity which allows an application of these techniques.
In fact, we use somewhat simpler anisotropic Sobolev spaces in 
our argument and provide an alternative approach to the 
meromorphic continuation of the resolvent~-- see \S\S\ref{on-forms}, \ref{stuff}.

\medskip\noindent\textbf{Remarks}. (i) If the coefficients of the generator of the flow are merely
$C^k$ for large enough $k$, then microlocal methods presented in this paper show that the Ruelle
zeta function can still be continued meromorphically to a strip $\{\Im\lambda\geq -k/C\}$, where
$C$ is a constant independent of $k$. That follows 
immediately from the fact that wavefront set
statements in $H^s$ regularity depend only on a finite number of derivatives of the symbols involved. In \cite{glp} a more precise estimate
on the width of the strip was provided. 

\noindent (ii) One conceptual difference between~\cite{glp} and the present paper is the following.
In~\cite[(2.11), (2.12)]{glp}, the resolvent
$(P-\lambda)^{-1}$ is decomposed into two pieces, one of which corresponds to resonances in a large disk
and the other one to the rest of the resonances; using an auxiliary determinant~\cite[(2.7)]{glp}, it is
shown that it is enough to study mapping properties of large iterates of $(P-\lambda)^{-1}$,
which implies that resonances outside the disk can be ignored in a certain asymptotic regime.
In our work, however, we show directly that $(P-\lambda)^{-1}$ lies in a class where one can take
the flat trace. In terms of the expression~\eqref{e:phoenix}, this requires uniform control
of the wavefront set of $\varphi^*_{-t}$ as $t\to +\infty$. Such a statement
does not follow from the analysis for bounded times and this is where the matters
are considerably simplified by using radial source/sink estimates originating in scattering theory.

\noindent (iii) In this paper we only provide analysis at bounded frequencies, but do not discuss the behavior
of $\zeta_R(\lambda)$ as $\lambda$
goes to infinity. However, a high frequency analysis of the zeta function is possible
using the methods of semiclassical analysis, which recover the structure
of $(P-\lambda)^{-1}$ modulo $\mathcal O(|\lambda|^{-\infty})$, rather than just compact, errors.
An example is provided by the bounds on the number of Pollicott--Ruelle resonances
in~\cite{fa-sj,ddz}.

\smallsection{Some further developments} Since this paper 
was first posted \arXiv{1306.4203} related results have appeared.
In \cite{DZ} the authors showed that Pollicott--Ruelle resonances
are the limits of eigenvalues of $ V/i + i\epsilon \Delta_g $,  
as $ \epsilon \to 0+ $, where $ -\Delta_g $ is any Laplace--Beltrami
operator on $ X $. In addition, for contact Anosov flows the spectral
gap is uniform with respect to $ \epsilon $. In \cite{JiZw}, Jin--Zworski
proved that for any Anosov flow there exists a strip with 
infinitely many resonances and a counting function which 
cannot be sublinear \eqref{eq:low}. For weakly mixing flows the estimate for
the size of that strip in terms of topological pressure 
was provided by Naud in the appendix to \cite{JiZw}. Guillarmou \cite{G1}
used the methods of \cite{fa-sj} and of this paper to 
study regularity properties of cohomological equations and to provide
applications.

Meromorphic continuation (of $ ( P - \lambda)^{-1} $ and of zeta 
fucntions) for flows on non-compact manifolds (or manifolds 
with boundary) with compact hyperbolic trapped sets was recently established
by Dyatlov--Guillarmou \cite{DG}. That required 
a development of new microlocal methods as the escape on the 
cotangent bundle can occur both at fiber infinity (as in this paper) {\em and} 
at the manifold infinity. A surprising application was given 
by Guillamou \cite{G2} who established 
deformation lens rigidity for a class of manifolds including manifolds with negative curvature and strictly convex boundary. That is the first result of that
kind in which trapping is allowed. 

\smallsection{Organization of the paper}
In \S \ref{prelim} we list the preliminaries from dynamical systems
and microlocal analysis. Precise definitions, references and proofs of the
statements in \S \ref{prelim} are given in the appendices. They are all standard and reasonably well 
known but as the paper is interdisciplinary in spirit we provide detailed arguments. Except for
references to texts \cite{ho1,ho3,ev-zw}, the paper is self-contained.

In \S \ref{micro} we simultaneously prove the meromorphic continuation
and describe the wave front set of the Schwartz
kernel of $ ( \mathbf P - \lambda)^{-1} $. This is based on results about
propagation of singularities. The vector field $ H_p $ has {\em radial-like
sets}, that is invariant conic closed sets which are
sources/sinks for the flow -- they correspond to 
stable/unstable directions in the Anosov decomposition. 
Away from those
sets the results are classical and due to Duistermaat--H\"ormander -- see for
instance
\cite[\S26.1]{ho3}. At the radial points we use the more recent
propagation results of Melrose \cite{mel} and Vasy \cite{vasy1}.
The a~priori regularity needed there is provided by the properties of 
the spaces $ H_{sG} $. 
Finally, 
in \S \ref{t1} we give our proof of the main theorem which is 
a straightforward application of the results in \S \ref{micro} and the
more standard results recalled in \S \ref{prelim}. 

\smallsection{Notation} We use the following notation: $ f =  \mathcal O_\ell ( g
 )_H $ means that
$ \|f \|_H  \leq C_\ell  g $ where the norm (or any seminorm) is in the
space $ H$, and the constant $ C_\ell  $ depends on $ \ell $. When either $ \ell $ or
$ H $ are absent then the constant is universal or the estimate is
scalar, respectively. When $ G = \mathcal O_\ell ( g )_{H_1\to H_2 } $ then
the operator $ G : H_1  \to H_2 $ has its norm bounded by $ C_\ell g $.

\section{Preliminaries}
\label{prelim}

\subsection{Dynamical systems}
\label{dyns}

Let $ X $ be a compact manifold and $ \varphi_t : X \to X $ be 
a $ C^\infty $ flow, $ \varphi_t = \exp t V $,  $ V \in
C^\infty ( X; T X) $.  The flow is {\em Anosov} if the tangent space 
to $ X $ has a continuous decomposition $ T_x X = E_0 ( x ) 
\oplus E_s (x) \oplus E_u ( x)$ which is invariant, 
$ d \varphi_t ( x ) E_\bullet ( x ) = E_\bullet ( \varphi_t ( x ) ) $,
$E_0(x)=\mathbb R V(x)$,
and for some $ C$ and $ \theta > 0 $ fixed 
\begin{equation}
\label{e:anosov}
\begin{split} 
&  | d \varphi_t ( x ) v |_{\varphi_t ( x )}  \leq C e^{ - \theta
  |t| } | v |_x  , \ \ v
\in E_u ( x ) , \ \ 
t < 0 , \\
& | d \varphi_t ( x ) v |_{\varphi_t ( x ) }  \leq C e^{ - \theta |t| } |  v |_x , \ \
v \in E_s ( x ) , \ \ 
t > 0 ,
\end{split} 
\end{equation}
where $ | \bullet |_y $ is given by a smooth Riemannian metric on $
X $. Note that we do not assume that the dimensions
of $E_u$ and $E_s$ are the same.

Fix a smooth volume form $\mu$ on $X$.
We present here some basic results:
an upper bound on the number of closed trajectories of $\varphi_t$
(Lemma~\ref{l:dyn-4}) and
on the volume of the set of trajectories that return to a small neighbourhood
of their originating point after a given time (Lemma~\ref{l:dyn-3}).
These bounds are used in the proof of Lemma \ref{l:appr}. See Appendix~\ref{s:dyn} for the proofs.
The constant
$L$ is defined in~\eqref{e:flow-der-bound}.
\begin{lemm}\label{l:dyn-3}
Define the following measure on $ X \times \RR $:  $\tilde\mu=\mu \times dt $ and
fix $t_e>0$. Then
there exists $C$ such that for each $\epsilon>0,T>t_e$,
and $n=\dim X$,
\begin{equation}
  \label{e:dyn-3}
\tilde\mu\big(\{(x,t)\mid t_e\leq t\leq T,\
d(x,\varphi_t(x))\leq \epsilon\}\big)
\leq C\epsilon^ne^{nLT}.
\end{equation}
\end{lemm}

In particular, by letting $\epsilon\to 0$, we get a bound on the
number of closed trajectories:
\begin{lemm}\label{l:dyn-4}
Let $N(T)$ be the number of closed trajectories of $\varphi_t$ of period no more
than $T$. Then
\begin{equation}
  \label{e:dyn-4}
N(T)\leq Ce^{(2n-1)LT}.
\end{equation}
\end{lemm}

\subsection{Trace identities}
\label{s:trace-identities}

Let $\varphi_t=e^{tV}$ be as in~\S\ref{dyns}
and  $\mathbf P:C^\infty(X;\mathcal E)\to C^\infty(X;\mathcal E)$ be defined
by $\mathbf P={1\over i}\mathcal L_V$ on the vector bundle of differential forms
of all orders on $X$, see~\eqref{e:the-P}. Let
$\mathcal E^k_0$ be the smooth invariant subbundle of $\mathcal E$ given by
all differential $k$-forms $\mathbf u$ satisfying $\iota_V\mathbf u=0$,
where $\iota$ denotes the contraction operator by a vector field~-- see also~\cite[(3.5)]{glp}.
We recall the trace formula of Guillemin \cite[Theorem 8,
(II.22)]{Gu} which is valid for any flow with nondegenerate periodic trajectories~-- see Appendix~\ref{s:guillemin}
for a self-contained proof in the Anosov case. In our notation it says that 
\begin{equation}
\label{eq:guill}   \tr^\flat e^{-i t\mathbf P } |_{ C^\infty ( X ;
\mathcal E^k_0 )} 
 =   \sum_{ \gamma }
\frac{ T_\gamma^\#   \tr ( \wedge^k \mathcal P_\gamma ) \,  \delta (
  t - T_\gamma ) } { | \det ( I - \mathcal P_\gamma ) 
  |} , \ \  t > 0 , \end{equation}
where $\gamma $'s  are periodic orbits,  $ \mathcal P_\gamma:=d\varphi_{-T_\gamma}|_{E_s\oplus E_u} $ is the
linearized  Poincar\'e map, $ T_\gamma $ is the period of $ \gamma $, and $
T_\gamma^\# $ is the primitive period. See~\S\ref{tft} for definition and properties
of the flat trace $\tr^\flat$.
By the Anosov property, and since we use negative times in the definition
of $\mathcal P_\gamma$, the eigenvalues of $\mathcal P_\gamma|_{E_u}$ satisfy $|\mu|<1$, therefore
$\det(I-\mathcal P_\gamma|_{E_u})>0$. Similarly $\det(I-\mathcal P_\gamma^{-1}|_{E_s})>0$.
If $E_s$ is orientable, then 
$\det(\mathcal P_\gamma|_{E_s})=\det(d\varphi_{-T_\gamma}|_{E_s})>0$;
since $\det(I-\mathcal P_\gamma|_{E_s})=\det(-\mathcal P_\gamma|_{E_s})\det(I-\mathcal P_\gamma^{-1}|_{E_s})$,
\[   
 | \det ( I - \mathcal P_\gamma ) | = ( -1 )^{ \dim E_s } \det ( I
- \mathcal P_\gamma ) , 
\]
that is \eqref{eq:Poinc} holds with $ q = \dim E_s $. We now assume~\eqref{eq:Poinc} for
some integer $ q $.

Consequently we relate the expressions on the right hand side of
\eqref{eq:guill} to the Ruelle zeta function using
\[  \det ( I - \mathcal P_\gamma ) = \sum_{ k=0}^{ n-1} ( - 1 )^k \tr
\wedge^k \mathcal P_\gamma .\]
This is a standard argument going back to Ruelle \cite{Ru1} but the 
particular determinants here seem to be rather different than the one
related to his transfer operators:
\begin{equation}
\label{eq:Rue}
\begin{split}
\zeta_{\rm R} ( \lambda ) & = \prod_{\gamma^\# } ( 1 - e^{ i \lambda T_\gamma^\#
} )= \exp \left( - \sum_{\gamma^\#} \sum_{ m= 1}^\infty \frac 1 m
  e^{ i \lambda m T_\gamma^\# } \right) \\
& = \exp \left( - \sum_{\gamma} { T^\#_\gamma  e^{ i \lambda
      T_\gamma} }/{ T_\gamma } \right) 
= \prod_{  k=0}^{n-1} \exp \left( -
\sum_\gamma  \frac{ T_\gamma^\# e^{ i \lambda T_\gamma
    } 
\tr \wedge^{ k} \mathcal P_\gamma
 }
{ T_\gamma |  \det(I - \mathcal P_\gamma)  | } \right)^{(-1)^ { k +  q }} \end{split}
\end{equation}
We note that thanks to Lemma \ref{l:dyn-4} the sums on the right hand
side converge for $ \Im \lambda \gg 1 $.

\subsection{Microlocal and semiclassical analyses}
\label{wfs}

In this section we present concepts and facts from microlocal/semiclassical
analysis which are needed in the proofs. Their proofs and detailed
references are provided in Appendix~\ref{a:wf}.

Let $ X $ be a manifold. For a distribution $ u\in\mathcal D'(X) $,
a phase space description of its singularities is given by the wave front
set $\WF(u)$, a closed conic subset of $T^*X\setminus 0$.
A more general object is the semiclassical wave front set defined
using a (small) asymptotic parameter $ h$ for $h$-tempered families of
distributions $ \{ u ( h ) \}_{ 0 <h < 1 }$:   $ \WFh ( u ) \subset
\overline T^* X $ where $ \overline T^* X $ is the fiber-radially compactified cotangent 
bundle, a manifold with interior $T^*X$ and boundary $\partial\overline T^*X=S^*X
=(T^*X\setminus 0)/\mathbb R^+$, the cosphere bundle. In addition to singularities, $\WFh$ measures oscillations on the
$h$-scale. 
The relation of the two wave front sets is the following:
if $u$ is an $h$-independent distribution, then
\begin{equation}
  \label{e:wf-wf-h}
\WF(u)=\WFh(u)\cap (T^*X\setminus 0), 
\end{equation}
see \S\ref{a:wf-2} and for a more general statement, \cite[(8.4.8)]{ev-zw}.

For operators we define the
wave front set $\WF'(B)$ (or $\WFh'(B)$ for $h$-dependent families of
operators) using the Schwartz kernel  -- see \eqref{e:wf'}.  This way
$ \WF' ( I ) = \Delta ( T^* X ) $, the diagonal in $ T^*X  \times T^*X
$, rather than $ N^* \Delta ( X ) $, the conormal bundle to the
diagonal in $ X \times X $. 

The following result, proved in~\S\ref{a:wf-2}, will allow us to calculate $ \WFh' ( ( P
-\lambda)^{-1} ) $, and thus, by~\eqref{e:wf-wf-h}, $\WF'( (P-\lambda)^{-1})$. It states that away from the fiber infinity,
the semiclassical wave front set of an operator is characterized using its
action on distributions:
\begin{lemm}
\label{l:wfs}
Let $B:C_{\rm{c}}^\infty(X)\to \mathcal D'(Y)$ be an $h$-tempered family of operators.
A point $(y,\eta,x,\xi)\in T^*(Y\times X)$
does not lie in $\WFh'(B)$ if and only if there exist neighbourhoods $U$ of $(x,\xi)$
and $V$ of $(y,\eta)$ such that
\begin{equation}
  \label{e:wf-char-op}
\WFh(f)\subset U\ \Longrightarrow\
\WFh(Bf)\cap V=\emptyset
\end{equation}
for each $h$-tempered family of functions
$f(h)\in C_{\rm{c}}^\infty(X)$.
\end{lemm}
We next state several semiclassical estimates used in \S \ref{micro}.
To be able to work with differential forms, we consider
a semiclassical pseudodifferential operator $\mathbf
P\in\Psi^k_h(X;\Hom(\mathcal E
))$
acting on $h$-tempered families of
distributions $\mathbf u(h)\in \mathcal D'(X;\mathcal E)$ with values in a vector bundle $\mathcal E$ over $X$.
For simplicity, we assume below that $X$ is a compact manifold.
We provide estimates in semiclassical Sobolev spaces $H_h^m(X,
\mathcal E )$ (denoted $ H_h^m $ for simplicity) and the corresponding
restrictions on wave front sets. Each of the estimates~\eqref{e:elliptic-est}, \eqref{e:hyperbolic-est},
\eqref{e:radial1-est}, \eqref{e:radial2-est}
is understood as follows: if the right-hand side is well-defined, then for $h$ small enough,
the left-hand side is well-defined and the estimate holds. For example, in the case
of~\eqref{e:hyperbolic-est}, if $\mathbf P\mathbf u\in H^m_h $ and
$B\mathbf u\in H^m_h $,
then we have $A\mathbf u\in H^m_h $. See \S\ref{a:wf-3} for the proofs.
\begin{prop}\label{l:elliptic} (Elliptic estimate) Let $\mathbf u(h)\in \mathcal 
D'(X;\mathcal E)$ be $h$-tempered.
Then:

1. If $A\in\Psi^0_h(X)$ (acting on $\mathcal D'(X;\mathcal E)$ diagonally)
and $\mathbf P$ is elliptic on $\WFh(A)$, then
for each $m$,
\begin{equation}
  \label{e:elliptic-est}
\|A\mathbf u\|_{H^m_h(X;\mathcal E)}\leq C\|\mathbf P \mathbf
u\|_{H^{m-k}_h(X;\mathcal E)}+\mathcal O(h^\infty). 
\end{equation}

2. If $\Ell_h(\mathbf P)\subset \overline T^*X$ denotes the elliptic set 
of $\mathbf P$, then
\begin{equation}
  \label{e:elliptic-wf}
\WFh(\mathbf u)\cap \Ell_h(\mathbf P)\subset\WFh(\mathbf P\mathbf u).
\end{equation}
\end{prop}
\begin{figure}
\includegraphics{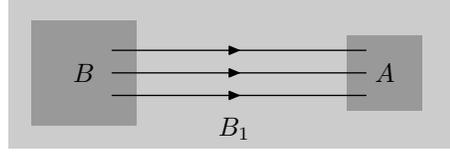}
\caption{The assumptions of Proposition~\ref{l:hyperbolic}, displaying the wave front
sets of $A,B,B_1$ and the flow lines of $H_p$.}
\label{f:hyperbolic}
\end{figure}
\begin{prop}\label{l:hyperbolic} (Propagation of singularities)
Assume that $\mathbf P\in\Psi^1_h(X;\Hom(\mathcal E
))$ and the semiclassical principal symbol
$$
\sigma_h(\mathbf P)\in S^1_h(X;\Hom(\mathcal E
))/hS^0_h(X;\Hom(\mathcal E
))
$$
is diagonal with entries%
\footnote{Strictly speaking, this means that $p-iq$ is some representative of the
equivalence class $\sigma_h(\mathbf P)$ satisfying the specified conditions.}
$p-iq$, with $p\in S^1(X;\mathbb R)$ independent of $h$
and $q\geq 0$ everywhere.
Assume also that $p$ is homogeneous of degree $1$ in $\xi$, for $|\xi|$ large enough.
Let $e^{tH_p}$ be the Hamiltonian flow of $p$ on $\overline T^*X$
and $\mathbf u(h)\in \mathcal D'(X;\mathcal E)$ be an $h$-tempered
family of distributions.
Then (see Figure~\ref{f:hyperbolic}):

1. Assume that $A,B,B_1\in\Psi_h^0(X)$ and for each $(x,\xi)\in\WFh(A)$,
there exists $T\geq 0$ with $e^{-TH_p}(x,\xi)\in\Ell_h(B)$
and $e^{tH_p}(x,\xi)\in \Ell_h(B_1)$ for $t\in [-T,0]$.
Then for each $m$,
\begin{equation}
  \label{e:hyperbolic-est}
\|A\mathbf u\|_{H^m_h(X;\mathcal E)}\leq C\|B\mathbf u\|_{H^m_h(X;\mathcal E)}
+Ch^{-1}\|B_1\mathbf P\mathbf u\|_{H^m_h(X;\mathcal E)}+\mathcal O(h^\infty).
\end{equation}

2. If $\gamma(t)$ is a flow line of $H_p$, then for each $T>0$,
\begin{equation}
  \label{e:hyperbolic-wf}
\gamma(-T)\not\in\WFh(\mathbf u),\
\gamma([-T,0])\cap\WFh(\mathbf P\mathbf u)=\emptyset\ \Longrightarrow\
\gamma(0)\not\in\WFh(\mathbf u).
\end{equation}
\end{prop}
Propagation of singularities states in particular
that if $\mathbf P\mathbf u=\mathcal O(h^\infty)_{C^\infty}$ and
$\mathbf u=\mathcal O(1)_{H^m_h}$ microlocally near some $(x,\xi)\in\overline T^*X$, then
$\mathbf u=\mathcal O(1)_{H^m_h}$ microlocally near $e^{tH_p}(x,\xi)$ for $t\geq 0$; in other words, regularity can be propagated
forward along the Hamiltonian flow lines. (If $q\leq 0$ instead, then regularity could be propagated backward.)
We next state less standard estimates guaranteeing
regularity of $\mathbf u$ near sources/sinks,
provided that $\mathbf u$ lies in a sufficiently high Sobolev space.

Denote by $ \kappa : T^* X \setminus 0
\to S^*X=\partial \overline T^* X $
the natural projection map.
Let $p$ be a real-valued function on $T^*X$; for simplicity, we assume that it is homogeneous of degree 1 in $\xi$.
Assume that $ L \subset T^* X \setminus 0 $ is a closed conic set
invariant under the flow $e^{tH_p}$ and
there exists an open conic neighbourhood $U$ of 
$ L$ with the following properties for some constant $\theta>0$:
\begin{equation}
\label{eq:defL}
\begin{split}
d\big( \kappa(e^{-tH_p}(U) ) , \kappa ( L ) \big) \to 0 
 \ &\text{ as $  t \to +\infty $;} \\
   (x,\xi)\in U   \ \Longrightarrow \ 
 |e^{-tH_p}(x,\xi)| \geq C^{-1} e^{\theta t} |\xi|,  \ &\text{ for any 
norm on the fibers.}
\end{split}
\end{equation}
We call $ L $ a {\em radial source}. A {\em radial sink} is 
defined analogously, reversing the direction of the flow. 
The following propositions come essentially  from the work of
Melrose~\cite[Propositions~9,10]{mel} and
Vasy~\cite[Propositions~2.3,2.4]{vasy1}.
The first one shows that
for sufficiently regular distributions
the wave front set at radial sources is controlled. 

\begin{figure}
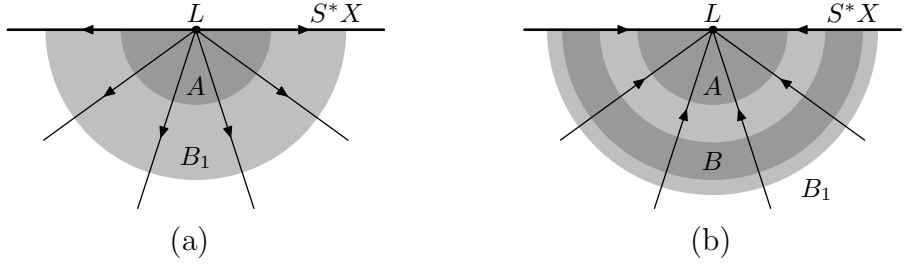

\includegraphics{zeta.3}
\qquad\qquad
\includegraphics{zeta.4}
\hbox to\hsize{\hss (a)\qquad\qquad\qquad\hss (b) \hss}
\caption{(a) The assumptions of Proposition~\ref{l:radial1}.
(b) The assumptions of Proposition \ref{l:radial2}. Here $S^*X$ is
the boundary of $\overline T^*X$ and the flow lines of $H_p$
are pictured.}
\label{f:radial}
\end{figure}
\begin{prop}\label{l:radial1}
Assume that $\mathbf P\in\Psi^1_h(X;\Hom(\mathcal E 
))$ is as in Proposition~\ref{l:hyperbolic}
and $L\subset T^*X\setminus 0$ is a radial source. Then there exists $m_0>0$ such that
(see Figure~\ref{f:radial}(a))

1. For each $B_1\in\Psi^0_h(X)$ elliptic on $\kappa(L)\subset S^*X=\partial \overline T^*X$, there exists $A\in\Psi^0_h(X)$
elliptic on $\kappa(L)$ such that if
$\mathbf u(h)\in\mathcal D'(X;\mathcal E)$ is $h$-tempered, then for each $m\geq m_0$,
\begin{equation}
  \label{e:radial1-est}
A\mathbf u\in H^{m_0}_h\ \Longrightarrow\
\|A\mathbf u\|_{H^m_h}\leq Ch^{-1}\|B_1\mathbf P \mathbf u\|_{H^m_h}+\mathcal O(h^\infty).
\end{equation}

2. If $\mathbf u(h)\in\mathcal D'(X;\mathcal E)$ is $h$-tempered and
 $B_1\in\Psi^0_h(X)$ is elliptic on $\kappa(L)$, then
\begin{equation}
  \label{e:radial1-wf}
B_1\mathbf u\in H^{m_0}_h,\
\WFh(\mathbf P\mathbf u)\cap \kappa(L)=\emptyset\ \Longrightarrow\
\WFh(\mathbf u)\cap\kappa(L)=\emptyset.
\end{equation}
\end{prop}

The second result shows that for sufficiently low
regularity we have a propagation result at radial sinks analogous to
\eqref{e:hyperbolic-est}.

%
\begin{prop}\label{l:radial2}
Assume that $\mathbf P\in\Psi^1_h(X;\Hom(\mathcal E 
))$ is as in Proposition~\ref{l:hyperbolic}
and $L\subset T^*X\setminus 0$ is a radial sink. Then there exists $m_0>0$ such that
for each $B_1\in\Psi^0_h(X)$ elliptic on $\kappa(L)$, there exists
$A\in\Psi^0_h(X)$ elliptic on $\kappa(L)$ and $B\in\Psi^0_h(X)$ with
$\WFh(B)\subset\Ell_h(B_1)\setminus \kappa(L)$, such that if $\mathbf u(h)\in\mathcal D'(X;\mathcal E)$
is $h$-tempered, then for each $m\leq -m_0$ (see Figure~\ref{f:radial}(b))
\begin{equation}
  \label{e:radial2-est}
\|A\mathbf u\|_{H^m_h}\leq C\|B\mathbf u\|_{H^m_h}+Ch^{-1}\|B_1\mathbf P\mathbf u\|_{H^m_h}
+\mathcal O(h^\infty).
\end{equation}
\end{prop}
\noindent\textbf{Remarks}. (i) In the case $q=0$, we can replace $\mathbf P$ by $-\mathbf P$
in Propositions~\ref{l:radial1} and~\ref{l:radial2} to make both of them apply to sources and sinks.

\noindent (ii) The precise value of the threshold $m_0$ can be computed
by being slightly more careful in the proofs (using a regularizer $\langle\epsilon\xi\rangle^{-\delta}$
for small $\delta>0$ in place of $\langle\epsilon\xi\rangle^{-1}$ and an additional
regularization procedure to justify~\eqref{e:hyperbolic-commutator})~-- see
for example~\cite[Propositions~2.3,2.4]{vasy1}.

\subsection{The flat trace}
\label{tft}

We now consider an operator $B:C^\infty(X)\to \mathcal D'(X)$
satisfying
\begin{equation}
  \label{e:no-diagonal}
\WF'(B)\cap\Delta(T^*X) 
=\emptyset,\quad
\Delta(T^*X):=\{(x,\xi,x,\xi)\mid (x,\xi)\in T^*X \}, 
\end{equation}
on a compact manifold $X$, and define the flat trace
\begin{equation}
  \label{e:flat-trace}
\tr^\flat B:= \int_X (\iota^* K_B)(x)\,dx , \quad
\iota:x\mapsto (x,x).
\end{equation}
Here $K_B$ is the Schwartz kernel of $X$ with respect to the density
$dx$ on $X$; the trace $\tr^\flat B$ does not depend on the choice
of the density.
The pullback $\iota^* K_B\in\mathcal D'(X)$ of the Schwartz kernel
$K_B\in\mathcal D'(X\times X)$ is defined under the condition~\eqref{e:no-diagonal}
as in \cite[Theorem~8.2.4]{ho1}.

To obtain a concrete expression for $ \tr^\flat B $ 
we use traces of regularized operators.
For that we introduce a family of mollifiers. 
Let $ d ( x, y ) $ be the geodesic distance 
for $ ( x, y ) $ in a neighbourhood of $ \Delta( X ) \subset 
X \times X$ with respect to some fixed Riemannian metric.
Let $ \psi \in C^\infty_{\rm{c}} ( \mathbb R , [0,1]) $ be equal to 1
near 0.  We define
$    E_\epsilon  : {\mathcal D}' ( X ) \to C^{\infty } ( X ) $, 
\begin{gather}
\label{eq:eeps2}
\begin{gathered} 
E_\epsilon u ( x ) = \int_X E_\epsilon ( x,y) u ( y )\, dy , \ \
E_\epsilon ( x, y ) =  \frac{ 1 } { F_\epsilon ( x )} \,
 \psi\left( \frac{ d ( x, y ) }{ \epsilon } 
\right)  , \end{gathered}
\end{gather} 
where 
$ F_\epsilon ( x ) $ is chosen so that $ E_\epsilon (1) = 1$ and
satisfies $ \epsilon^n / C \leq F_\epsilon( x ) \leq C 
\epsilon^n $. We have
\begin{equation}
\label{eq:eeps1}   E_\epsilon \in \Psi^{-\infty } ( X ) , \ \ 
E_\epsilon \longrightarrow I \ \text{ in $ \Psi^{0+} ( X ) $.} 
\end{equation}

The next lemma shows that the flat trace is well approximated by regular
traces -- see \S\ref{a:wf-1} for a proof.
\begin{lemm}
\label{l:appr1}
For $B$ satisfying~\eqref{e:no-diagonal} and $E_\epsilon$ given by
\eqref{eq:eeps2}  we have
\begin{equation}
\label{eq:appr1}
\tr^\flat B=\lim_{\epsilon\to 0} \tr E_\epsilon B E_\epsilon 
\end{equation}
where the trace on the right hand side is well-defined since 
$E_{\epsilon} BE_{\epsilon} $ is smoothing and thus trace class
on $L^2(X)$.
\end{lemm}

If an operator $\mathbf B$ instead acts on sections of a smooth vector bundle,
$\mathbf B:C^\infty(X;\mathcal E)\to \mathcal D'(X;\mathcal E)$, and satisfies~\eqref{e:no-diagonal},
then we can define the trace of $\mathbf B$ by the formula
$$
\tr^\flat\mathbf B=\tr^\flat\sum_{j=1}^r B_{jj},\quad
\mathbf B(f\mathbf e_l)=\sum_{j=1}^r (B_{jl}f)\mathbf e_j,\
f\in C^\infty(X),
$$
if $\mathbf e_1,\dots,\mathbf e_r$ is a local frame of $\mathcal E$ and $\mathbf B$
is supported in the domain of the local frame~-- the general case is handled by a partition
of unity and the independence of the choice of the frame is easily verified.

\section{Properties of the resolvent}
\label{micro}

In this section we use the anisotropic Sobolev spaces $ H_{sG} $ and
the propagation results recalled in \S \ref{wfs} to describe the 
microlocal structure of the meromorphic continuation of the resolvent.
Our proof is different that the argument in~\cite{fa-sj} in the sense
that we use a less refined weight to define anisotropic Sobolev
spaces and derive the Fredholm property of $\mathbf P-\lambda$
from propagation of singularities.

Anisotropic Sobolev spaces appeared in the study
of Anosov flows in the works of Baladi \cite{bal},
Baladi--Tsujii
\cite{b-t},  Gou\"ezel--Liverani \cite{g-l}, Liverani \cite{liv2},
and other authors. However, the use of microlocally defined
exponential weights allows a more direct study using PDE methods.

\subsection{Anisotropic Sobolev spaces}
\label{on-forms}

Let $(X,\varphi_t)$ be as
in \S\ref{dyns} and consider the vector bundle, 
$\mathcal E$,
of differential forms of all orders on $X$. (The resolvents on forms
of different degree are decoupled from each other, however we treat
them as a single resolvent to simplify notation.)
Consider the first order differential operator
\begin{equation}
  \label{e:the-P}
\mathbf P:C^\infty(X;\mathcal E)\to C^\infty(X;\mathcal E),\quad
\mathbf P(\mathbf u)={1\over i}\mathcal L_V \mathbf u, \quad
\mathcal E: =\bigoplus_{j=0}^n \Lambda^j ( T^*X ) ,
\end{equation}
where $V$ is the generator of the flow $\varphi_t$,
$\mathcal L$ denotes the Lie derivative, and $\mathbf u$
is a differential form on $X$.

The principal symbol $\sigma(\mathbf P)=p\in S^1(X;\mathbb R)$, as defined in~\S\ref{a:wf-1}, is diagonal
and homogeneous of degree 1:
$
p(x,\xi)=\xi( V(x) ) $, 
$(x,\xi)\in T^*X$.
This follows immediately from the fact that for any basis $\mathbf e_1,\dots,\mathbf e_r$ of
$\mathcal E$, and all $u_1,\dots,u_r\in C^\infty(X)$,
$$
\mathcal L_V\sum_{j=1}^r u_j \mathbf e_j=\sum_{j=1}^r Vu_j \, \mathbf e_j
+\sum_{j=1}^r u_j \, \mathcal L_V\mathbf e_j,
$$
where the second term in the sum is a differential operator of order 0.

The Hamilton flow is 
$ e^{ t H_p } ( x , \xi ) = ( \varphi_t ( x ) , ( {}^T
d\varphi_t ( x) )^{-1} \xi ) $. Define the decomposition
$$T_x^*X=E_0^*(x)\oplus E_s^*(x)\oplus E_u^*(x),$$ where
$E_0^*(x),E_s^*(x),E_u^*(x)$ are dual
to $E_0(x),E_u(x),E_s(x)$.
From \eqref{e:anosov} it follows that
\begin{equation}
\label{eq:dualan}
\begin{gathered}
\xi\not\in E_0^*(x)\oplus E_s^*(x)\ \Longrightarrow\
d\big(\kappa(e^{tH_p}(x,\xi)),\kappa(E_u^*)\big)\to 0\text{ as }t\to
+\infty, \\
\xi\not\in E_0^*(x)\oplus E_u^*(x)\ \Longrightarrow\ d\big(\kappa(e^{tH_p}(x,\xi)),\kappa(E_s^*)\big)\to 0\text{ as }t\to -\infty.
\end{gathered}
\end{equation}
 Here $
\kappa : T^* X \setminus 0 \to S^* X $ is the projection defined before \eqref{eq:defL}.
Moreover, under the assumptions of~\eqref{eq:dualan} we have $|e^{tH_p}(x,\xi)|\geq C^{-1}e^{\theta|t|}|\xi|$,
and the convergence in~\eqref{eq:dualan} and the constant $C$ are locally uniform
in $(x,\xi)$.
In particular \eqref{eq:dualan} implies that, in the sense of definition
\eqref{eq:defL}, the closed conic sets $ E_s^* $ and $ E_u^* $ are a
radial source and a radial sink, respectively~-- see Figure~\ref{f:dynamics} below.

Anisotropic Sobolev spaces have a long tradition in microlocal analysis
going back to the work of Duistermaat \cite{duistermaat} and 
Unterberger \cite{unterberger}. 
To define a version on which  $\mathbf
P-\lambda$ is a Fredholm  operator, 
we use a function $ m_G\in C^\infty(T^*X\setminus 0;[-1,1]) $, homogeneous of degree $ 0 $ and such that
\begin{gather}
  \label{e:m-G}
\begin{gathered} 
m_G=1 \quad\text{near }E_s^*,\ \ \
m_G=-1 \quad\text{near }E_u^*,\\
H_p m_G\leq 0 \quad\text{everywhere}.
\end{gathered}
\end{gather}
A function with these properties, supported
in a small neighbourhood of $E_s^*\cup E_u^*$,
can be constructed using part~1 of Lemma~\ref{l:radialesc}. A more
refined version, not needed here, can be found in~\cite[Lemma~1.2]{fa-sj}.
With $ m_G $ in place we choose a pseudodifferential operator
$G\in\Psi^{0+}(X)$ satisfying
\begin{equation}
  \label{e:sigma-G}
\sigma(G)(x,\xi)=m_G(x,\xi)\log|\xi|,
\end{equation}
where $|\cdot|$ is any smooth norm on the fibers of $T^*X$. 
Then, using \cite[\S\S 8.3,9.3,14.2]{ev-zw} as in \cite[(3.9)]{ddz}, $ \exp(\pm s G)\in \Psi^{s+}(X)$ for any $
s > 0 $. 
The anisotropic Sobolev spaces are defined using this
exponential weight:
$$
H_{sG}:=\exp(-sG)(L^2(X) ),\quad
\|\mathbf u\|_{H_{sG}}:=\|\exp(sG)\mathbf u\|_{L^2}.
$$
Note that $H^s(X)\subset H_{sG}\subset H^{-s}(X)$.
This is because the symbol of $\exp(\pm s G)$ lies in the
class $S^s_{1-\varepsilon,\varepsilon}$ for each $\varepsilon>0$,
see~\cite[(18.1.1)${}''$]{ho3}, and thus maps $H^{k}(X)\to H^{k-s}(X)$ for
each $k$, see~\cite[Theorem~18.1.13]{ho3}.

Define the domain, $D_{sG}$, of $ \mathbf P $ as the set of $ \mathbf u \in H_{
  sG } $ such that the distribution $ \mathbf P u $ is in $ H_{ s G }
$. The Hilbert space norm on $ D_{ sG} $ is given by 
$
\|\mathbf u\|_{D_{sG}}^2:=\|\mathbf u\|_{H_{sG}}^2+\|\mathbf P \mathbf
u\|_{H_{sG}}^2$. 

\subsection{Ruelle--Pollicott resonances for forms}
\label{stuff}

Here we state the properties of the resolvent of $\mathbf P$:
\begin{prop}
\label{l:resolvent-meromorphic}
Fix a constant $C_0>0$. Then for $s>0$ large enough depending on $C_0$,
$\mathbf P-\lambda:D_{sG}\to H_{sG}$ is a Fredholm operator of index
0 in the region $\{\Im\lambda>-C_0\}$.
\end{prop}
%
\begin{prop}
  \label{l:our-stuff-actually-makes-sense}
Let $s>0$ be fixed as in Proposition~\ref{l:resolvent-meromorphic}.
Then there exists a constant $C_1$ depending on $s$, such that
for $\Im\lambda>C_1$, the operator $\mathbf P-\lambda:D_{sG}\to H_{sG}$ is invertible
and
\begin{equation}
\label{e:our-stuff-actually-makes-sense}
(\mathbf P-\lambda)^{-1}=i\int_0^\infty e^{i\lambda t}\varphi_{-t}^*\,dt,
\end{equation}
where $\varphi_{-t}^*:C^\infty(X;\mathcal E)\to C^\infty(X;\mathcal E)$ is the pullback operator
by $\varphi_{-t}$ on differential forms and the integral on the right-hand side converges in
operator norm $H^s\to H^s$ and $H^{-s}\to H^{-s}$.
\end{prop}
The Fredholm property and the invertibility of $\mathbf P-\lambda$ for
large $\Im\lambda$
show that the resolvent $\mathbf R(\lambda)=(\mathbf P-\lambda)^{-1}:H_{sG}\to H_{sG}$
is a meromorphic family of operators with poles of finite rank~-- see for example~\cite[Proposition~D.4]{ev-zw}.
Note that Ruelle--Pollicott resonances, the poles of $\mathbf R(\lambda)$ in the region
$\Im\lambda>-C_0$, are then the poles of the meromorphic continuation of the Schwartz kernel
of the operator given by the right-hand side of~\eqref{e:our-stuff-actually-makes-sense},
and thus are independent of the choice of $s$ and the weight $G$.
Microlocal structure of $\mathbf R(\lambda)$ is described in
\begin{prop}
  \label{l:resolvent-properties}
Let $ C_0 $ and $ s $ be as in
Proposition~\ref{l:resolvent-meromorphic} and assume
$\Im\lambda_0>-C_0$. Then for $\lambda$ near $\lambda_0$,
\begin{equation}
\label{eq:Laurent}
\mathbf R(\lambda)=\mathbf R_H(\lambda) - \sum_{j=1}^{J(\lambda_0)} { ( \mathbf P -
  \lambda_0)^{j-1} \Pi \over (\lambda-\lambda_0)^j}
\end{equation}
where $\mathbf R_H(\lambda)$
holomorphic near $\lambda_0$, $ \Pi : H_{ sG } \to H_{sG} $ is the commuting
projection onto the kernel of $ ( \mathbf P -\lambda_0)^{J(\lambda_0)} $, and 
\begin{equation}
  \label{e:houston}
\WF'(\mathbf R_H(\lambda))\subset \Delta(T^*X)\cup \Omega_+\cup (E_u^*\times E_s^*),\quad
\WF'(\Pi )\subset E_u^*\times E_s^*,
\end{equation}
where $\Delta(T^*X)$ is the diagonal and $\Omega_+$ is the positive flow-out of $e^{tH_p}$ on $\{p=0\}$:
$$
\Omega_+=\{(e^{tH_p}(x,\xi),x,\xi)\mid t\geq 0,\ p(x,\xi)=0\}.
$$
\end{prop}

In~\S\ref{s:micro-1}, we construct a semiclassical nontrapping parametrix and study
its $h$-wave front set. In~\S\ref{s:micro-2}, we express $\mathbf R(\lambda)$
via the parametrix and use the results of~\S\ref{s:micro-1} to finish the proofs
of Propositions~\ref{l:resolvent-meromorphic}--\ref{l:resolvent-properties}.

\subsection{Complex absorbing potential near the zero section}
\label{s:micro-1}
We will modify $ \mathbf P - \lambda $ by a complex absorbing potential which
will eliminate trapping and guarantee invertibility of the modified operator.

It is convenient now to introduce a semiclassical parameter $h$ and use the algebra
$\Psi_h$ of semiclassical pseudodifferential operators, see~\S\ref{a:wf-2}.
If $\mathbf P$ is defined in~\eqref{e:the-P}, then
$h\mathbf P\in\Psi^1_h(X;\Hom(\mathcal E 
))$ is a semiclassical differential operator
with principal symbol $p=\sigma_h(h\mathbf P)$.

The original operator $\mathbf P$ is independent of $h$. However, the parameter $h$
enters in the parametrix $\mathbf R_\delta(z)$ defined in Proposition~\ref{l:parametrix-properties} below,
which is a convenient tool to show the Fredholm property of $\mathbf P-\lambda$. Moreover,
the semiclassical wavefront set of $\mathbf R_\delta(z)$ can be computed by studying the dependence
of $\WFh(\mathbf R_\delta(z)\mathbf f)$ on $\WFh(\mathbf f)$; this is not possible
for nonsemiclassical wavefront sets as we lose information on how the lengths
of covectors in $\WF(\mathbf f)$ and $\WF((\mathbf P-\lambda)^{-1}\mathbf f)$ are related.
Therefore, semiclassical methods are convenient for the proof of Proposition~\ref{l:resolvent-properties},
which is the key component of the present paper.

We need a semiclassical adaptation, $G(h)\in \Psi^{0+}_h(X)$, of the operator $G$,
such that
\begin{equation}
\label{eq:DefGh}
\sigma_h(G(h))(x,\xi)=(1-\chi(x,\xi))m_G(x,\xi)\log|\xi|,
\end{equation}
where $\chi\in C_0^\infty(T^*X)$ is equal to 1 near the zero section,
and $\WFh(G(h))$ does not intersect the zero section.
Note that, since $H_p\log|\xi|$ is homogeneous of degree zero,
\begin{equation}
  \label{eq:posGh}
H_p\sigma_h(G(h))(x,\xi)=(H_pm_G(x,\xi))\log|\xi|+\mathcal O(1)_{S^0_h}.
\end{equation}
Define the space
$H_{sG(h)}=\exp(-sG(h))(L^2(X))$. For each fixed $h>0$, the
operator $G(h)$ lies in $\Psi^{0+}(X)$ and
$\sigma(G(h))(x,\xi)=\sigma_h(G(h))(x,h\xi)$; therefore,
$\sigma(G(h)-G)$ is bounded as $|\xi|\to\infty$. By~\cite[Theorem~8.8]{ev-zw},
$H_{sG(h)}=H_{sG}$ and
the norms are equivalent, with the constant depending on $h$.
We also use the semiclassical analogue of the space $D_{sG}$, with
the norm
$$
\|\mathbf u\|_{D_{sG(h)}}^2:=\|\mathbf u\|_{H_{sG(h)}}^2+\|h\mathbf P \mathbf u\|_{H_{sG(h)}}^2.
$$

We modify $h\mathbf P$ by adding an $h$-pseudodifferential \emph{complex
absorbing potential} $-iQ_\delta\in\Psi^0_h(X)$, which provides a localization
to a neighbourhood of the zero section:
$$
\WFh(Q_\delta)\subset \{|\xi|<\delta\},\quad
\sigma_h(Q_\delta)>0\text{ on }\{|\xi|\leq\delta/2\},\quad
\sigma_h(Q_\delta)\geq 0\text{ everywhere},
$$
here $|\cdot|$ is a fixed norm on the fibers of $T^*X$.
The action of 
\[ \mathbf P_{\delta }( z)  := h\mathbf P-iQ_\delta-z \]
 on $H_{sG}$ is equivalent
to the action on $L^2$ of the conjugated operator
\[\begin{split}
\mathbf P_{\delta,s}(z)&:=e^{sG(h)} \mathbf P_\delta ( z ) 
e^{-sG(h)} 
=
\mathbf P_\delta ( z ) +s[G(h),h\mathbf P]+\mathcal O(h^2)_{\Psi^{-1+}_h},
\end{split}\]
where the asymptotic expansion follows from \cite[\S\S
8.3,9.3,14.2]{ev-zw} -- see~\cite[(3.11)]{ddz}.
We note that  $[G(h),Q_\delta]=\mathcal O(h^\infty)_{\Psi^{-\infty}}$
for small enough $\delta$,
because $\WFh(G(h))$ does not intersect the zero section.

\begin{figure}
\includegraphics{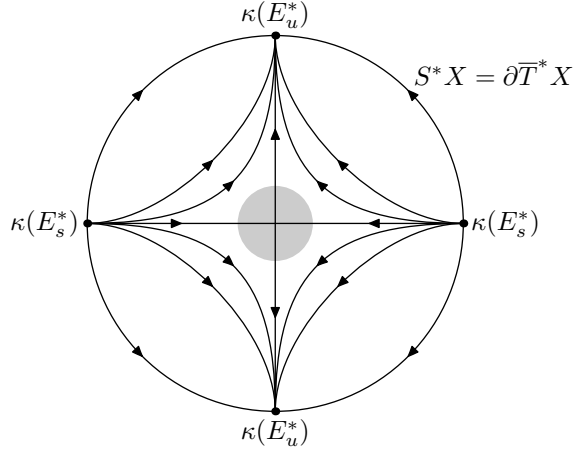}
\caption{Dynamics of the flow $e^{tH_p}$ on $\{p=0\}=\overline{E_s^*\oplus E_u^*}\subset \overline T^*X$,
projected onto the fibers of $\overline{T^*X}$. The shaded region is
the wave front set of $Q_\delta$.}
\label{f:dynamics}
\end{figure}

We now use the propagation of semiclassical singularities and 
the elimination of trapping due to the complex absorbing potential
to establish existence and properties of the inverse of $ \mathbf P_\delta ( z ) $. The relation between propagation and solvability
has a long tradition -- see \cite[\S26.1]{ho3}
Although the 
details below may look complicated the idea is simple and natural,
given the dynamics of the flow pictured on Figure~\ref{f:dynamics}:
given bounds on $\|\mathbf P_\delta(z)\mathbf u\|_{H_{sG(h)}}$, we first establish
bounds on $\mathbf u$ microlocally near the sources $\kappa(E_s^*)$
by Proposition~\ref{l:radial1}. By ellipticity (Proposition~\ref{l:elliptic})
we can also estimate $\mathbf u$ on $\{p\neq 0\}$ and in $\{|\xi|<\delta/2\}$,
where the latter is made possible by the potential $Q_\delta$. The resulting
estimates can be propagated forward
along the flow $e^{tH_p}$, using Proposition~\ref{l:hyperbolic}, to the whole
$\overline T^*X\setminus \kappa(E_u^*)$; finally, to bound $\mathbf u$
microlocally near $\kappa(E_u^*)$, we use Proposition~\ref{l:radial2}.
The spaces $ H_{sG(h)} $ provide
the correct regularity for Propositions~\ref{l:radial1}
and~\ref{l:radial2}.

\begin{prop}
  \label{l:parametrix-properties}
Fix a constant $C_0>0$ and $ \varepsilon > 0 $. Then for $s>0$ large
enough depending on $C_0$
and $h$ small enough, the operator
$$
\mathbf P_\delta ( z ) :D_{sG(h)}\to H_{sG(h)},\quad
-C_0h\leq \Im z\leq 1,\quad
|\Re z|\leq h^{\varepsilon},
$$
is invertible, and
the inverse, $\mathbf R_\delta(z)$, satisfies
$$
\|\mathbf R_\delta(z)\|_{H_{sG(h)}\to H_{sG(h)}}\leq Ch^{-1},\quad
\WFh'(\mathbf R_\delta(z))\cap T^*(X\times X)\subset \Delta(T^*X)\cup\Omega_+,
$$
with $\Delta(T^*X),\Omega_+$ defined in Proposition~\ref{l:resolvent-properties},
and 
$\WFh'(
\bullet)\subset\overline T^*(X\times X)$
is defined in~\S\ref{a:wf-2}.
\end{prop}
\begin{proof}
We first prove the bound
\begin{equation}
  \label{e:parametrix-bound-1}
\|\mathbf u\|_{H_{sG(h)}}\leq Ch^{-1}\|\mathbf f\|_{H_{sG(h)}},\quad
\mathbf u\in D_{sG(h)},\quad
\mathbf f=\mathbf P_\delta (z)\mathbf u.
\end{equation}
Without loss of generality, we assume that $\|\mathbf u\|_{H_{sG(h)}}\leq 1$.
By a microlocal partition of unity, it suffices to obtain bounds on $A\mathbf u$, where
$A\in\Psi^0_h(X)$ falls into one of the following five cases:

\noindent\textbf{Case 1}: $\WFh(A)\cap \{p=0\}\cap \{|\xi|\geq \delta/2\}=\emptyset$. Then
$\mathbf P_{\delta,s}(z)$ is elliptic on $\WFh(A)$. We have
$\|A\mathbf u\|_{H_{sG(h)}}=\|A^se^{sG(h)}\mathbf u\|_{L^2}$,
where $A^s=e^{sG(h)}Ae^{-sG(h)}\in\Psi^0_h$ and $\WFh(A^s)\subset\WFh(A)$.
By Proposition~\ref{l:elliptic},
$$\|A^se^{sG(h)}\mathbf u\|_{L^2}\leq C\|B_1^s\mathbf P_{\delta,s}(z) e^{sG(h)}\mathbf u\|_{L^2}+\mathcal O(h^\infty),$$
where $B_1^s\in\Psi^0_h(X)$ is microlocalized in a neighbourhood of $\WFh(A)$.
Putting $B_1:=e^{-sG(h)}B_1^s e^{sG(h)}$, we obtain
\begin{equation}
\label{e:case-1}
\|A\mathbf u\|_{H_{sG(h)}}\leq C\|B_1\mathbf f\|_{H_{sG(h)}}+\mathcal O(h^\infty).
\end{equation}

\noindent\textbf{Case 2}: $\WFh(A)$ is contained in a small neighbourhood
of $\kappa(E_s^*)$, where $\kappa:T^*X\setminus 0\to S^*X=\partial\overline T^*X$
is the natural projection. By~\cite[Theorem~8.6]{ev-zw}, $\exp(sG(h))\in\Psi^s_h(X)$
and $\sigma_h(\exp(sG(h)))=\exp(s\sigma_h(G(h)))=|\xi|^s$ near
$\kappa(E_s^*)$. Therefore, $H_{sG(h)}$ is microlocally equivalent to the space
$H_h^s(X;\mathcal E)$ near $\kappa(E_s^*)$
in the sense that
\begin{equation}
\label{eq:snorm}
\|B\mathbf v\|_{H_h^s}\leq C\|\mathbf v\|_{H_{sG(h)}}+\mathcal O(h^\infty),\quad
\|B\mathbf v\|_{H_{sG(h)}}\leq C\|\mathbf v\|_{H_h^s}+\mathcal O(h^\infty),
\end{equation}
for each $B\in\Psi^0_h(X)$ with $\WFh(B)$ contained in a neighbourhood
of $\kappa(E_s^*)$ and each $h$-tempered $\mathbf v$.

Since $\Im z\geq -C_0h$, we get $\Im\sigma_h ( \mathbf P_\delta (z))
\leq 0$. The set 
$E_s^*$ is a radial source (see the discussion
following~\eqref{eq:dualan}) and we can apply
Proposition~\ref{l:radial1}
and 
\eqref{eq:snorm} to obtain, for $s$ sufficiently large, 
\begin{equation}
\label{e:case-2}
\|A\mathbf u\|_{H_{sG(h)}}\leq
Ch^{-1}\|B_1\mathbf f\|_{H_{sG(h)}}+\mathcal O(h^\infty),
\end{equation}
where $B_1\in\Psi^0_h(X)$ is some operator with $\WFh(B_1)$ in a neighbourhood
of $\kappa(E_s^*)$.

\noindent\textbf{Case 3}: $\WFh(A)$ is contained in a small neighbourhood
of some $(x_0,\xi_0)\in \{p=0\}\setminus \overline{E_u^*}$,
where $\overline{E_u^*}=E_u^*\cup\kappa(E_u^*)$ is the closure of $E_u^*$ in $\overline T^*X$. Then
by~\eqref{eq:dualan} and the discussion
following it, $d(e^{tH_p}(x_0,\xi_0),\kappa(E_s^*))\to 0$ in $\overline T^*X$
as $t\to -\infty$. Therefore, for any fixed neighbourhood $U$ of $\kappa(E_s^*)$,
there exists $B\in\Psi^0_h(X)$ with $\WFh(B)\subset U$ and $T>0$ such that
$e^{-TH_p}(\WFh(A))\subset \Ell_h(B)$.

From \eqref{e:m-G},\eqref{eq:posGh} and the fact that $\Im z\geq
-C_0h$, 
$$\Im\sigma_h(\mathbf P_{\delta,s}(z))=-\sigma_h(Q_\delta)-\Im
z+shH_p\sigma_h(G(h))\leq 0, \ \
\text{in $S^1_h(X)/hS^0_h(X)$.} $$
Applying Proposition~\ref{l:hyperbolic} to the operator $\mathbf P_{\delta,s}(z)$ and arguing similarly
to Case~1, we get $\|A\mathbf u\|_{H_{sG(h)}}\leq C\|B\mathbf u\|_{H_{sG(h)}}
+Ch^{-1}\|B_2\mathbf f\|_{H_{sG(h)}}+\mathcal O(h^\infty)$, where
$B_2\in\Psi^0_h$ is microlocalized in a small neighbourhood of
$\bigcup_{t\in [-T,0]}e^{tH_p}(\WFh(A))$. Now, $\|B\mathbf u\|_{H_{sG(h)}}$ can be estimated
by Case~2, yielding
\begin{equation}
\label{e:case-3}
\|A\mathbf u\|_{H_{sG(h)}}\leq Ch^{-1}(\|B_1\mathbf f\|_{H_{sG(h)}}
+\|B_2\mathbf f\|_{H_{sG(h)}})+\mathcal O(h^\infty),
\end{equation}
where $B_1\in\Psi^0_h(X)$ is microlocalized in a small neighbourhood of $\kappa(E_s^*)$.

\noindent\textbf{Case 4}: $\WFh(A)$ is contained in a small neighbourhood
of some $(x_0,\xi_0)\in E_u^*$. Then $e^{tH_p}(x_0,\xi_0)$ converges to the zero section
as $t\to -\infty$; therefore, there exists $T>0$ such that
$e^{-TH_p}(\WFh(A))\subset\{|\xi|<\delta/2\}$. Similarly to Case~3, by propagation of singularities
we find $\|A\mathbf u\|_{H_{sG(h)}}\leq C\|B\mathbf u\|_{H_{sG(h)}}
+Ch^{-1}\|B_2\mathbf f\|_{H_{sG(h)}}+\mathcal O(h^\infty)$, where $\WFh(B)\subset \{|\xi|<\delta/2\}$
and $\WFh(B_2)$ is contained in a small neighbourhood of $\bigcup_{t\in [-T,0]}e^{tH_p}(\WFh(A))$.
Estimating $\|B\mathbf u\|_{H_{sG(h)}}$ by Case~1, we get
\begin{equation}
  \label{e:case-4}
\|A\mathbf u\|_{H_{sG(h)}}\leq Ch^{-1}(\|B_1\mathbf f\|_{H_{sG(h)}}
+\|B_2\mathbf f\|_{H_{sG(h)}})+\mathcal O(h^\infty),
\end{equation}
where $B_2$ is microlocalized in a small neighbourhood of $e^{-TH_p}(\WFh(A))$.

\noindent\textbf{Case 5}: $\WFh(A)$ is contained in a small neighbourhood
of $\kappa(E_u^*)$. Note that the space $H_{sG(h)}$ is microlocally equivalent to the
space $H_h^{-s}(X)$ near $\kappa(E_u^*)$, similarly to Case~2.
Since $E_u^*$ is a radial sink, by Proposition~\ref{l:radial2} we get,
for $s$ sufficiently large,
$
\|A\mathbf u\|_{H_{sG(h)}}\leq
C\|B\mathbf u\|_{H_{sG(h)}}+Ch^{-1}\|B_1\mathbf f\|_{H_{sG(h)}}
+\mathcal O(h^\infty)$,
where $B,B_1\in\Psi^0_h(X)$ are microlocalized in a small neighbourhood of $\kappa(E_u^*)$
and $\WFh(B)\cap\kappa(E_u^*)=\emptyset$. 
Then $\|B\mathbf u\|_{H_{sG(h)}}$ can be estimated by a combination of the preceding
cases, using a microlocal partition of unity; this gives
\begin{equation}
  \label{e:case-5}
\|A\mathbf u\|_{H_{sG(h)}}\leq
Ch^{-1}\|\mathbf f\|_{H_{sG(h)}}+\mathcal O(h^\infty).
\end{equation}
Combining~\eqref{e:case-1}, \eqref{e:case-2}--\eqref{e:case-5}, we get~\eqref{e:parametrix-bound-1}.

For  for the dynamics of $-H_p$, $E_s^*$ is a sink and $E_u^*$ a source. Hence
the proof of \eqref{e:parametrix-bound-1} applies
 to $ - \mathbf P_\delta ( z )^* = 
- ( h\mathbf P - i Q_\delta - z)^* $,
and we obtain the adjoint bound
\begin{equation}
  \label{e:parametrix-bound-2}
\|\mathbf v\|_{H_{-sG(h)}}\leq Ch^{-1}\|\mathbf P_\delta(z)^*\mathbf v\|_{H_{-sG(h)}},\quad
\mathbf v\in D_{-sG(h)}.
\end{equation}

We now show that $\mathbf P_\delta(z)$ is invertible $D_{sG(h)}\to H_{sG(h)}$. Injectivity follows
immediately from~\eqref{e:parametrix-bound-1}; we also get the bound on the inverse
once surjectivity is proved. To see surjectivity, note first that~\eqref{e:parametrix-bound-1}
implies that if $\mathbf u_j\in D_{sG(h)}$ and
$\mathbf P_\delta(z) \mathbf u_j$ is a Cauchy sequence in $H_{sG(h)}$, then
$\mathbf u_j$ is a Cauchy sequence in $H_{sG(h)}$ as well; since the operator $\mathbf P_\delta(z)$
is closed on $H_{sG(h)}$ with domain $D_{sG(h)}$, we see that
the image of $\mathbf P_\delta(z)$ is a closed subspace
of $H_{sG(h)}$. Now, $H_{-sG(h)}$ is the dual to $H_{sG(h)}$ under the $L^2$ pairing
(fixing an inner product on the fibers of $\mathcal E$)~-- see 
\cite[(8.3.11)]{ev-zw}. Therefore, it suffices to show that if
$\mathbf v\in H_{-sG(h)}$ and $\langle \mathbf P_\delta(z)\mathbf u,\mathbf v\rangle_{L^2}=0$
for all $\mathbf u\in D_{sG(h)}$, then $\mathbf v=0$. Taking
$\mathbf u\in C^\infty$, we see that $\mathbf P_\delta(z)^*\mathbf v=0$; it remains
to use~\eqref{e:parametrix-bound-2}.

To show the restriction on the wave front set
of $\mathbf R_\delta(z)$, by Lemma~\ref{l:wfs} it is enough to show that for each
$(y,\eta,x,\xi)\in T^*(X\times X)\setminus (\Delta(T^*X)\cup\Omega_+)$,
there exist neighbourhoods $U$ of $(x,\xi)$ and $V$ of $ (y,\eta)$ such that
for each
$h$-tempered $\mathbf u\in H_{sG(h)}$ and $\mathbf f:=(h\mathbf P-iQ_\delta-z)\mathbf u$,
if $\WFh(\mathbf f)\subset U$, then $\WFh(\mathbf u)\cap V=\emptyset$.
This follows similarly to the proof of part~2 of Proposition~\ref{l:elliptic}
from the estimates~\eqref{e:case-1},\eqref{e:case-3},\eqref{e:case-4},
keeping in mind that $\kappa(E_s^*\cup E_u^*)\cap T^*X=\emptyset$.
\end{proof}

\subsection{Proofs of Propositions~\ref{l:resolvent-meromorphic}--\ref{l:resolvent-properties}}
\label{s:micro-2}

We assume that $\lambda$ varies in some compact subset of $\{\Im\lambda>-C_0\}$
and choose $h$ small enough so that $z=h\lambda$ satisfies
$-C_0h\leq\Im z\leq 1$, $|\Re z|\leq h^{1/2}$.

Proposition~\ref{l:resolvent-meromorphic} follows immediately from
Proposition~\ref{l:parametrix-properties}, given that
$H_{sG},D_{sG}$ are topologically isomorphic to $H_{sG(h)},D_{sG(h)}$
and $Q_\delta:D_{sG}\to H_{sG}$ is smoothing and thus compact.

To show Proposition~\ref{l:our-stuff-actually-makes-sense},
we first note that since derivatives of the flow $\varphi_t$
are bounded exponentially in $t$, we have
$\varphi_t^*=\mathcal O(e^{C_1|t|})_{H^{\pm s}\to H^{\pm s}}$,
where $C_1$ is a constant depending on $s$.
Therefore, if $\Im\lambda>C_1$, $\mathbf u\in H_{sG}\subset H^{-s}$, and
$(\mathbf P-\lambda)\mathbf u=\mathbf f\in H_{sG}$, then we see
$$
\mathbf u=-\int_0^\infty \partial_t(e^{i\lambda t}\varphi_{-t}^*\mathbf u)\,dt=i\int_0^\infty e^{i\lambda t}\varphi_{-t}^*\mathbf f\,dt,
$$
where the integrals converge in $H^{-s}$.
This implies that $\mathbf P-\lambda$ is injective $D_{sG}\to H_{sG}$
and thus invertible, and~\eqref{e:our-stuff-actually-makes-sense} holds.

For \eqref{eq:Laurent} in Proposition \ref{l:resolvent-properties} 
we note that the Fredholm property shows that, near a pole $ \lambda_0
$, $ \mathbf R ( \lambda ) = \mathbf R_H ( \lambda ) + \sum_{ j=1}^{J
  ( \lambda_0 ) } A_j / ( \lambda - \lambda_0 )^j $, where $ A_j $ are
operators of finite rank -- see for instance \cite[\S D.3]{ev-zw}. We
have
\begin{equation}
\label{eq:proj} \Pi := - A_1 = \frac1{2 \pi i}   \oint_{\lambda_0 } (
\lambda -  \mathbf P  )^{-1} d \lambda, 
\end{equation}
 $ [ \Pi , \mathbf P ] = 0 $ and, using Cauchy's theorem, $ \Pi^2 =
 \Pi $. 
Equating powers of $ \lambda - \lambda_0 $ in 
the equation $ ( \mathbf P - \lambda ) \mathbf R ( \lambda) = I_{
  H_{sG} } $ shows that $ A_j = - ( \mathbf P - \lambda_0 )^{j-1}  \Pi $, and 
$ ( \mathbf P - \lambda_0)^{ J( \lambda_0 ) } \Pi= 0 $. 

Finally, to show \eqref{e:houston} we use the formula
\begin{equation}
  \label{e:omar-bleu}
\mathbf R(\lambda)=h\big(\mathbf R_\delta(z)
-i\mathbf R_\delta(z)Q_\delta \mathbf R_\delta(z)\big)
-\mathbf R_\delta(z)Q_\delta \mathbf R(\lambda)Q_\delta \mathbf R_\delta(z),
\end{equation}
where $\mathbf R(\lambda)=(\mathbf P-\lambda)^{-1}$,
$\mathbf R_\delta(z)=(h\mathbf P-z-iQ_\delta)^{-1}$,
and $z=h\lambda$. Now, by Proposition~\ref{l:parametrix-properties},
and since $Q_\delta$ is pseudodifferential, 
we get
$$
\WFh'(\mathbf R_\delta(z)-i\mathbf R_\delta(z)Q_\delta \mathbf R_\delta(z))\cap T^*(X\times X)
\subset \Delta(T^*X)\cup\Omega_+.
$$
To handle the remaining term in~\eqref{e:omar-bleu}, we first assume that
$\lambda$ is not a pole of $\mathbf R$.
Applying again Proposition~\ref{l:parametrix-properties}, we see that
$$
\begin{gathered}
\WFh'(\mathbf R_\delta(z)Q_\delta \mathbf R(\lambda)Q_\delta \mathbf R_\delta(z))
\cap T^*(X\times X)
\subset\Upsilon_\delta,\\\Upsilon_\delta:= \{(\rho',\rho)\mid \exists t,s\geq 0:
e^{tH_p}(\rho)\in\WFh(Q_\delta),\
e^{-sH_p}(\rho')\in\WFh(Q_\delta)\}.
\end{gathered}
$$
Therefore, $
\WFh'(\mathbf R(\lambda))\cap T^*(X\times X)\subset \Delta(T^*X)\cup\Omega_+\cup \Upsilon_\delta
$.
Since $\mathbf R(\lambda)$ does not depend on $\delta$ and $h$, by~\eqref{e:wf-wf-h}, 
$$
\WF'(\mathbf R(\lambda))\subset \Delta(T^*X)\cup\Omega_+\cup \bigcap_{\delta>0}\Upsilon_\delta
=\Delta(T^*X)\cup\Omega_+\cup (E_u^*\times E_s^*),
$$
as claimed.

In a neighbourhood of a pole $\lambda_0$ of $\mathbf R$, we replace
$\mathbf R(\lambda)$ in~\eqref{e:omar-bleu} by $(\lambda-\lambda_0)^{J(\lambda_0)}\mathbf R(\lambda)$.
Arguing as before, we get $\WF'((\lambda-\lambda_0)^{J(\lambda_0)}\mathbf R(\lambda))
\subset \Delta(T^*X)\cup\Omega_+\cup (E_u^*\times E_s^*)$ uniformly in $\lambda$ near $\lambda_0$.
By taking $J(\lambda_0)$
derivatives at $\lambda=\lambda_0$ we obtain the first part of~\eqref{e:houston}.
By taking $J(\lambda_0)-1$ derivatives at $\lambda=\lambda_0$,
we get
$\Pi=-\mathbf R_\delta(z_0)Q_\delta\Pi Q_\delta \mathbf R_\delta(z_0)$,
which implies the second part of~\eqref{e:houston}.

\section{Proof of the main theorem}
\label{t1}

The proof is based on \eqref{e:our-stuff-actually-makes-sense} which
relates the resolvent and the propagator. The description of the wave
front set of $ ( \mathbf P - \lambda )^{-1} $ allows us to take the
flat trace of the left hand side composed with $ \varphi_{-t_0}^* $
and that formally gives the meromorphic continuation.

To justify this we first use the mollifiers $ E_\epsilon $ to obtain trace class
operators to which Lemma \ref{l:appr1} can be applied:
\begin{lemm}
\label{l:appr}
Suppose that $ E_\epsilon $ is given by \eqref{eq:eeps2} 
and that $ T \geq t_0 > 0 $.
Then there exists a constant $ C $, independent of $ \epsilon,T $ 
such that 
\begin{gather}
\label{eq:eeps}  
\begin{gathered}
\| E_\epsilon \varphi_{-T}^*  E_\epsilon \|_{\tr}  \leq C e^{CT}\epsilon^{-n-2}  \
\ \text{ and } \ \ 
 \int_T^{T+1} |\tr E_\epsilon \varphi_{-t}^*  E_\epsilon|  \,dt  \leq C e^{C T}.
\end{gathered}
\end{gather}
\end{lemm}
\begin{proof}
We replace $\varphi_{-t}^*$ by $\varphi_t^*$ (considering the flow in the opposite time direction).
The first estimate follows from
\[ \begin{split} 
\|  E_\epsilon \varphi_T^*  E_\epsilon \|_{\tr}  & \leq
\| E_\epsilon \|_{\tr}\cdot \| \varphi_T^* \|_{ L^2 \to L^2} \cdot
\| E_\epsilon \|_{L^2 \to L^2} 
\\ & 
\leq C e^{CT}\| E_\epsilon \|_{\tr} \leq Ce^{CT} \| ( - \Delta_g + 1)^{-k} 
( - \Delta_g + 1 )^k E_\epsilon \|_{\tr} \\
& \leq Ce^{CT} \| ( - \Delta_g + 1)^{-k} \|_{\tr}\cdot
\| ( - \Delta_g + 1 )^k E_\epsilon \|_{L^2 \to L^2} 
\leq C'e^{CT} \epsilon^{-2k} , 
\end{split} \]
provided $ 2k > n $. Here $g$ is any fixed Riemannian metric on $X$.
For the second estimate in \eqref{eq:eeps} we
use the definition of $ E_\epsilon$:
\[ \begin{split} 
\int_T^{T+1}  |\tr E_\epsilon \varphi_t^* E_\epsilon|  \, dt 
& = 
\int_T^{T+1}  \int_{X\times X} 
 E_\epsilon ( x, y ) E_\epsilon 
( \varphi_t ( y ) , x ) \, dx  dy  dt 
\\
& \leq C \epsilon^{-2n} \int_{ T }^{T+1} 
\int_{X\times X} \bbbone_{ \{  d ( x , y ) < c_1 \epsilon \}} 
\bbbone_{ \{ d ( x, \varphi_t ( y) ) < c_1 \epsilon \} }  \, dx\, dy\, dt  
\\
& \leq C \epsilon^{-n} \int_{ T }^{T+1} \int_X 
\bbbone_{ \{ d ( y, \varphi_t ( y) ) < 2c_1 \epsilon \} }  \,
dy\, dt \leq C' e^{nLT}, 
\end{split} \]
where the last estimate comes from  Lemma \ref{l:dyn-3}.
\end{proof}

We now complete the proof of the meromorphic continuation 
of $ \zeta_R ( \lambda ) $. Thanks to formula \eqref{eq:Rue} we need
to show that 
\begin{equation}
\label{eq:LHS}
f_k ( \lambda ):= \frac 1 i \sum_\gamma  \frac{ T_\gamma^\# e^{ i \lambda T_\gamma
    } 
\tr \wedge^{ k} \mathcal P_\gamma
} { | \det( I - \mathcal P_\gamma)  | }  = \frac {\partial}  {\partial \lambda } 
\log \exp 
\left( - 
\sum_\gamma  \frac{ T_\gamma^\# e^{ i \lambda T_\gamma
    } 
\tr \wedge^{ k} \mathcal P_\gamma
 }
{ T_\gamma |  \det(I - \mathcal P_\gamma)  | } \right)\end{equation}
has a meromorphic continuation to $ \Im \lambda > - C_0 $ for any $ C_0 $,
with poles that are simple and residues which are integral.

Fix $ t_0 $ such that $0 < t_0 <T_\gamma$ for all $\gamma$ and put 
 $ \mathbf P_k := \mathbf P |_{ C^\infty ( X ; \mathcal E^k_0 )} $ 
where
$\mathcal E^k_0$ is defined in~\S\ref{s:trace-identities}.
For large $T>0$, take $\chi_T\in C_0^\infty(t_0/2,T+1)$
such that $\chi=1$ near $[t_0,T]$ and $|\chi|\leq 1$ everywhere.
Integrating~\eqref{eq:guill}
against the function $\chi_T(t)e^{i\lambda t}$, we get
$$
{1\over i}\sum_\gamma {\chi_T(T_\gamma)T_\gamma^\# e^{i\lambda T_\gamma}\tr\wedge^k\mathcal P_\gamma  \over |\det(I-\mathcal P_\gamma)|}
={1\over i}\tr^\flat\int_0^\infty \chi_T(t)e^{it(\lambda-\mathbf P_k)}\,dt.
$$
Using the bound on the number of closed geodesics given in 
Lemma \ref{l:dyn-4} together with~\eqref{eq:appr1}, we see that for $\Im\lambda\gg 1$,
$$
\begin{aligned}
f_k(\lambda)&={1\over i}\lim_{T\to +\infty} \tr^\flat\int_0^\infty \chi_T(t)e^{it(\lambda-\mathbf P_k)}\,dt\\
&={1\over i}\lim_{T\to +\infty}\lim_{\varepsilon\to 0}\tr\int_{t_0}^\infty \chi_T(t)E_\varepsilon e^{it(\lambda-\mathbf P_k)}E_\varepsilon\,dt\\
&={1\over i}\lim_{\varepsilon\to 0}\lim_{T\to +\infty}\tr\int_{t_0}^\infty \chi_T(t)E_\varepsilon e^{it(\lambda-\mathbf P_k)}E_\varepsilon\,dt
\end{aligned}
$$
We can change the order in which limits are taken by~\eqref{eq:eeps};
we can replace the domain of integration by $(t_0,\infty)$ since
$\tr E_\varepsilon e^{-it\mathbf P_k} E_\varepsilon=0$ for $\varepsilon$ small enough
and $t\in [t_0/2,t_0]$.

Let $\mathbf R_k(\lambda)=\mathbf R(\lambda)|_{H_{sG}(X;\mathcal E^k_0)}$, where $\mathbf R(\lambda)$
is the inverse of $\mathbf P-\lambda$ on the anisotropic Sobolev space $H_{sG}(X;\mathcal E)$,
studied in~\S\ref{stuff}, and $s$ is large depending on $C_0$. By
Proposition~\ref{l:our-stuff-actually-makes-sense}, we have for $\Im\lambda\gg 1$,
$$
f_k(\lambda)=-\lim_{\varepsilon\to 0}\tr E_\varepsilon e^{it_0(\lambda-\mathbf P_k)}\mathbf R_k(\lambda)E_\varepsilon.
$$
Because of the choice of $ t_0$ ($ 0 < t_0 < T_\gamma $ for all $ \gamma $),
and as $\WF'(e^{-it_0\mathbf P_k})$ is contained in the graph of $e^{t_0H_p}$,
Proposition~\ref{l:resolvent-properties} shows that $ e^{ - i t_0 \mathbf P_k } 
\mathbf R_k(\lambda) $ satisfies the assumptions of Lemma~\ref{l:appr1}
with the poles handled as in~\eqref{eq:Laurent}. Hence, by another
application of~\eqref{eq:appr1},
\[
f_k ( \lambda ) = -e^{ i \lambda t_0}  \tr^\flat \left(
e^{ -i t_0 \mathbf P_k} \mathbf R_k(\lambda)\right) ,
\]
which is a meromorphic function.
Finally, to see that $f_k$ has simple poles and integral residues, we use the following
elementary result based on the fact that traces of nilpotent operators
are $ 0 $:
\begin{lemm}
\label{l:Fred}
Suppose that 
that a linear map $ A : {\mathbb C}^m \to {\mathbb C}^m $ satisfies  $ (A -
\lambda_0)^{ J }  = 0 $ for some $\lambda_0\in\mathbb C$. Then for 
$ \varphi $  holomorphic near $ \lambda_0 $ we have 
\[    \lim_{ \lambda\to \lambda_0 } ( \lambda - \lambda_0 )  \tr
\left( \varphi (
A ) \sum_{ j=1}^{J}  \frac { ( A - \lambda_0 ) ^{j-1}  }
{ (\lambda - \lambda_0)^{j} } \right) = m\varphi(\lambda_0 ) ,\]
where $ \varphi ( A ) $ is defined by the power series expansion at  $
\lambda_0 $ (which is finite). 
\end{lemm}

From \eqref{eq:Laurent} we have near a pole $\lambda_0$ of $\mathbf R_k$,
\[
e^{  i t_0 \lambda  }  e^{ -i t_0 \mathbf P_k}\mathbf R_k(\lambda)  = 
 e^{ i t_0 \lambda } 
\mathbf R_{H,k} ( \lambda) -
e^{ i t_0 \lambda} \sum_{ j=1}^{ J( \lambda_0, k ) }
\frac{ e^{ -i t_0 \mathbf P_k} ( \mathbf P_k - \lambda_0 )^{j-1} \Pi_k }{ ( \lambda - \lambda_0 )^j }   ,
\]
where $ \mathbf R_{H,k} $ is holomorphic  near $
\lambda_0$ and $ \Pi_k $ is given by \eqref{eq:proj}:
\[  \Pi_k :=\frac 1 {2 \pi i } \oint_{\lambda_0} ( \lambda - \mathbf P_k
 )^{-1} d \lambda, \ \ \  
  \tr^\flat \Pi_k = \tr_{ H_{sG} } \Pi_k  \in \NN . \]
Here we use the fact that $ \tr^\flat $ and $ \tr_{ H_{sG} } $ agree on finite
rank operators (as follows from an approximation statement and the fact that the trace of a smoothing operator
is the integral of its Schwartz kernel over the diagonal, see~\eqref{e:get-trace-back}).
We now apply Lemma \ref{l:Fred} with $ \varphi ( \mu ) = e^{ -i t_0
  \mu}  $ and $  A = \mathbf P_k |_{ \ker ( \mathbf P_k - \lambda_0)^{J} }
$.

\appendix

\section{Estimates on recurrence}
\label{s:dyn}

In this Appendix we provide proofs of statements made in 
\S \ref{dyns}. 

It follows immediately from the Anosov property~\eqref{e:anosov} that
(with $I$ denoting the identity operator)
\begin{equation}
  \label{e:dyn-0}
t\neq 0,\
\varphi_t(x)=x\ \Longrightarrow\
(d\varphi_t(x)-I)|_{E_u(x)\oplus E_s(x)}\quad\text{is invertible}.
\end{equation}
Indeed, if $v\in E_u(x)\oplus E_s(x)$ and $d\varphi_t(x)v=v$, then
$d\varphi_{Nt}(x)v=v$ for all $N\in\mathbb Z$, implying by~\eqref{e:anosov}
that $v=0$.

The following
lemma is a generalization of~\eqref{e:dyn-0} to the case when $\varphi_t(x)$ is close
to $x$.
We fix a smooth distance function $d(\cdot,\cdot)$ on $X$
and a smooth norm $|\cdot|$ on the fibers of $TX$.
\begin{lemm}\label{l:dyn-1}
Let $\delta_0>0$ and
$ \mathcal T_{x,y}:T_xX\to T_y X,\quad
d(x,y) < \delta_0,
$
be a continuous family of invertible linear transformations such that
$\mathcal T_{x,x}=I$ and $\mathcal T_{x,y}$ maps
$E_u(x),E_s(x),\mathbb RV(x)$ onto $E_u(y),E_s(y),\mathbb RV(y)$.
Fix $t_e>0$. Then there exist $\delta\in (0,\delta_0)$ and $C$
such that
\begin{equation}
  \label{e:dyn-1}
|v|\leq C|(d\varphi_t(x)-\mathcal T_{x,\varphi_t(x)})v|\quad\text{if }
d(x,\varphi_t(x))<\delta,\
t\geq t_e,\
v\in E_u(x)\oplus E_s(x).
\end{equation}
\end{lemm}
\begin{proof}
We first note that it suffices to prove~\eqref{e:dyn-1} for
sufficiently large $t$. Indeed, if $N$ is a large fixed integer,
$v\in E_u(x)\oplus E_s(x)$, and $d(x,\varphi_t(x))$ and $|(d\varphi_t(x)-\mathcal T_{x,\varphi_t(x)})v|$ are both
small, then $d(x,\varphi_{Nt}(x))$ and $|(d\varphi_{Nt}(x)-\mathcal T_{x,\varphi_{Nt}(x)})v|$
are small as well; applying~\eqref{e:dyn-1} for $Nt$ in place of $t$, we get that $|v|$ is small.

Assume that the conditions of~\eqref{e:dyn-1} are satisfied and
put $v=v_u+v_s$, where $v_u\in E_u(x),v_s\in E_s(x)$.
For $t$ large enough, the Anosov property \eqref{e:anosov} implies
$$
|v_u|\leq \textstyle{1\over 2}|d\varphi_t(x)v_u|,\quad
|d\varphi_t(x)v_s|\leq \textstyle{1\over 2}|v_s|;
$$
since for $\delta$ small enough, $\|\mathcal T_{x,\varphi_t(x)}\|,\|\mathcal T_{x,\varphi_t(x)}^{-1}\|$
are close to 1, we get
$$
\begin{gathered}
|v|\leq |v_u|+|v_s|\leq 3\big(|(d\varphi_t(x)-\mathcal T_{x,\varphi_t(x)})v_u|
+|(d\varphi_t(x)-\mathcal T_{x,\varphi_t(x)})v_s|\big)
\\\leq C |(d(\varphi_t(x)-\mathcal T_{x,\varphi_t(x)})v|,
\end{gathered}
$$
where the last inequality is due to the fact that
$(d\varphi_t(x)-\mathcal T_{x,\varphi_t(x)})v_u\in E_u(\varphi_t(x))$,
$(d\varphi_t(x)-\mathcal T_{x,\varphi_t(x)})v_s\in E_s(\varphi_t(x))$.
\end{proof}
Fix a constant $L>0$ such that for some choice of the norm on the space
$C^2(X)$ of twice differentiable functions on $X$, there exists a constant $C$ such that
\begin{equation}
  \label{e:flow-der-bound}
\|f\circ\varphi_t\|_{C^2(X)}\leq Ce^{L|t|}\|f\|_{C^2(X)},\quad
f\in C^2(X).
\end{equation}
Such $L$ exists since $X$ is compact and $\varphi_t$ is a one-parameter group.
As a consequence of~\eqref{e:flow-der-bound}
(since it gives a bound on the Lipschitz norm of $\varphi_t$), we get
\begin{equation}
  \label{e:flow-dist-bound}
d(\varphi_t(x),\varphi_t(x'))\leq Ce^{L|t|}d(x,x').
\end{equation}

The next lemma in particular implies (by letting $\epsilon\to 0$) that
two different closed trajectories of nearby periods $t,t'$ have to be
at least $\delta e^{-Lt}$ away from each other, where $\delta$ is a small constant.
\begin{lemm}\label{l:dyn-2}
Fix $t_e>0$. Then there exist $C,\delta>0$ such that for each $\epsilon>0$,
\begin{equation}
\begin{gathered}
d(x,\varphi_t(x))\leq \epsilon,\
d(x',\varphi_{t'}(x'))\leq \epsilon,\
t,t'\geq t_e,\
|t-t'|\leq \delta,\
d(x,x')\leq\delta e^{-Lt}\\
\Longrightarrow\
|t-t'|\leq C\epsilon,\
\exists s\in (-1,1): d(x,\varphi_s(x'))\leq C\epsilon.
\end{gathered}
\end{equation}
\end{lemm}
\begin{proof}
Without loss of generality, we may assume that $\epsilon$ is small depending on $\delta$.
By~\eqref{e:flow-dist-bound}, we see that $d(\varphi_t(x),\varphi_t(x''))\leq C\delta$
whenever $d(x,x'')\leq\delta e^{-Lt}$. Therefore, we may operate in a coordinate neighbourhood
containing $x,x',\varphi_t(x),\varphi_{t'}(x')$, identified with a ball in $\mathbb R^n$.
We replace $x'$ with $\varphi_s(x')$ for some $|s|<1$ so that
\begin{equation}
  \label{e:transverse}
x'-x\in E_u(x)\oplus E_s(x).
\end{equation}
By~\eqref{e:flow-der-bound}, we have for all $j,k$,
$$
|\partial^2_{x_jx_k}\varphi_t (x'')|\leq Ce^{Lt}\quad\text{for }
d(x,x'')\leq \delta e^{-Lt};
$$
using the Taylor expansion of $\varphi_t(x)$ in $x$, we see that
$$
|\varphi_t(x')-\varphi_t(x)-d\varphi_t(x)(x'-x)|\leq Ce^{Lt}|x'-x|^2\leq C\delta |x'-x|.
$$
Next, $|\partial^2_t\varphi_t(x')|\leq C$; by Taylor expanding $\varphi_t(x')$ in $t$, we get
$$
|\varphi_{t'}(x')-\varphi_t(x')-V(\varphi_t(x'))(t'-t)|\leq C|t'-t|^2\leq C\delta|t'-t|.
$$
Together, these give
$$
|\varphi_{t'}(x')-\varphi_t(x)-d\varphi_t(x)(x'-x)-V(\varphi_t(x'))(t'-t)|\leq C\delta(|x'-x|+|t'-t|).
$$
Since $d(x,\varphi_t(x))\leq\epsilon$ and $d(x',\varphi_{t'}(x'))\leq\epsilon$, we get
$$
|(d\varphi_t(x)-I)(x'-x)+V(\varphi_t(x'))(t'-t)|\leq C\delta(|x'-x|+|t'-t|)+C\epsilon.
$$
Let $\mathcal T_{x,y}$ be a family of transformations satisfying the conditions of Lemma~\ref{l:dyn-1};
it can be defined for example using parallel transport along geodesics with respect to some Riemannian metric
and projectors corresponding to the decomposition $TX=E_0\oplus E_u\oplus E_s$.
Then $\mathcal T_{x,y}$
maps $E_u(x)\oplus E_s(x)$ onto $E_u(y)\oplus E_s(y)$. Since
$d(x,\varphi_t(x))\leq \epsilon$, we get for $\epsilon$ small enough
depending on $\delta$,
$|(I-\mathcal T_{x,\varphi_t(x)})(x'-x)|\leq \delta|x-x'|$.
Since $|\varphi_t(x')-\varphi_t(x)|\leq C\delta$, we find
$|V(\varphi_t(x'))-V(\varphi_t(x))|\leq C\delta$.
Then
$$
|(d\varphi_t(x)-\mathcal T_{x,\varphi_t(x)})(x'-x)+V(\varphi_t(x))(t'-t)|
\leq C\delta(|x'-x|+|t'-t|)+C\epsilon.
$$
Now, by~\eqref{e:transverse},
$(d\varphi_t(x)-\mathcal T_{x,\varphi_t(x)})(x'-x)\in E_u(\varphi_t(x))\oplus E_s(\varphi_t(x))$;
since this space is transverse to $V(\varphi_t(x))$, and by Lemma~\ref{l:dyn-1}, we get
$$
|x'-x|+|t'-t|\leq C(|(d\varphi_t(x)-\mathcal T_{x,\varphi_t(x)})(x'-x)|+|t'-t|)\leq C\delta(|x'-x|+|t'-t|)+C\epsilon.
$$
It remains to choose $\delta$ small enough so that $C\delta<1/2$.
\end{proof}
We now give a volume bound on the set of nearly closed trajectories:

\begin{proof}[Proof of Lemma \ref{l:dyn-3}]
First of all, we can replace the range of values of $t$ in~\eqref{e:dyn-3} by
$|t-T|\leq\delta/2$, where $\delta$ is the constant from Lemma~\ref{l:dyn-2}.
(Indeed, we can write $[t_e,T]$ as a union of such intervals.)
Next, let $x_1,\dots,x_N$, with $N$ depending on $T$, be a maximal
set of points in $X$ such that $d(x_j,x_k)\geq\delta e^{-LT}/2$. Since
the metric balls of radius $\delta e^{-LT}/4$ centered at $x_j$
are disjoint, by calculating the volume of their union we find
$N\leq Ce^{nLT}$. Now,
$$
\begin{gathered}
\{(x,t)\mid |t-T|\leq\delta/2,\
d(x,\varphi_t(x))\leq \epsilon\}\subset
\bigcup_{j=1}^N A_j,\\
A_j:=\{(x,t)\mid |t-T|\leq \delta/2,\
d(x,x_j)\leq \delta e^{-LT}/2,\
d(x,\varphi_t(x))\leq\epsilon\}.
\end{gathered}
$$
Take some $j$ such that $A_j$ is nonempty and fix $(x',t')\in A_j$.
Then for each $(x,t)\in A_j$, we have
$|t-t'|\leq\delta$, $d(x,x')\leq\delta e^{-LT}$. By Lemma~\ref{l:dyn-2},
$A_j$ is contained in an $\mathcal O(\epsilon)$ sized tubular neighbourhood of
the trajectory $\{(\varphi_s(x'),t')\mid |s|<1\}$. Therefore, we get
$\tilde\mu(A_j)\leq C\epsilon^n$, finishing the proof.
\end{proof}

\begin{proof}[Proof of Lemma \ref{l:dyn-4}]
Let $\gamma(t)=\varphi_t(x_0)$ be a closed trajectory of period $t_0$. Then
for each $\epsilon>0$, we have by~\eqref{e:flow-dist-bound},
\begin{equation}
  \label{e:dyn-4-int}
d(x,\varphi_t(x))\leq C\epsilon\quad\text{if }
|t-t_0|\leq \epsilon\text{ and }
d(x,\gamma(s))\leq \epsilon e^{-Lt_0}\quad\text{for some }s.
\end{equation}
Moreover, for $t_0\leq T$ and $\epsilon$ small enough depending on $T$,
the tubular neighbourhoods on the right-hand side of~\eqref{e:dyn-4-int} for different closed trajectories
do not intersect. The volume (in $x,t$) of each tubular neighbourhood
is bounded from below by $C^{-1}\epsilon^n e^{-(n-1)Lt_0}$;
it remains to let $\epsilon\to 0$ and apply Lemma~\ref{l:dyn-3}.
\end{proof}

\section{Proof of Guillemin's trace formula}
\label{s:guillemin}

In this appendix, we give a self-contained proof of Guillemin's trace formula~\eqref{eq:guill}
(including the special case~\eqref{eq:ABG})
in the case of Anosov flow $\varphi_t=e^{tV}$ on a compact manifold $X$.
The proof is somewhat simplified by the fact that $E_u(x)\oplus E_s(x)$ is a subbundle
of $TX$ transversal to $\mathbb R V$ and invariant under the flow.

If $\gamma(t)=\varphi_t(x_0)$ is a closed trajectory with period $t_0\neq 0$
(here $t_0$ need not be the \emph{primitive} period), then 
the linearized Poincar\'e map is defined by
\begin{equation}
  \label{e:p-gamma}
\mathcal P_\gamma:=d \varphi_{-t_0}(x_0)|_{E_u(x_0)\oplus E_s(x_0)}.
\end{equation}
Note that $I-\mathcal P_\gamma$ is invertible by~\eqref{e:dyn-0}. The maps
$d\varphi_{-t_0}(\varphi_s(x_0))$ are conjugate to each other by $d\varphi_s(x_0)$ for all
$s$, therefore the expressions $\det(I-P_\gamma)$ and $\tr(\wedge^k P_\gamma)$,
used in~\eqref{eq:guill}, are independent of the choice of the base point on $\gamma$.

Fix a density $dx$ on $X$ and let $K(t,y,x)$ be the Schwartz kernel of
$\varphi_{-t}^*=e^{-itP}$ with respect to this density, that is for $f\in C^\infty(X)$,
\begin{equation}
  \label{e:defK}
f(\varphi_{-t}(y))=\int_X K(t,y,x)f(x)\,dx.
\end{equation}
To be able to define the flat trace of $\varphi_{-t}^*$ as a distribution in $t\in \mathbb R\setminus 0$,
we need to take some $\chi(t)\in C_{\mathrm c}^\infty(\mathbb R\setminus 0)$
and show that the operator
$$
T_\chi:=\int_{\mathbb R} \chi(t)\varphi_{-t}^*\,dt
$$
satisfies the condition~\eqref{e:no-diagonal}, that is $\WF'(T_\chi)$ does not intersect the diagonal.
By the formula for the wave front set of a pushforward~\cite[Theorem~8.2.12]{ho1},
we know that
$$
\WF'(T_\chi)\subset \{(y,\eta,x,-\xi)\mid \exists t\in\supp \chi: (t,0,y,\eta,x,\xi)\in\WF(K)\},
$$
and thus it suffices to show that
\begin{equation}
  \label{e:lullaby}
\WF(K)\cap \{(t,0,x,\xi,x,-\xi)\mid t\neq 0,\ (x,\xi)\in T^*X\setminus 0\}=\emptyset.
\end{equation}
Note that~\eqref{e:lullaby} is exactly the condition under which one can define the pullback
$K(t,x,x)\in\mathcal D'((\mathbb R\setminus 0)\times X)$ of $K$ by the map
$(t,x)\mapsto (t,x,x)$, and
$$
\tr^\flat (T_\chi)=\int_{\mathbb R\times X}\chi(t) K(t,x,x)\,dx dt.
$$
Now, $K(t,y,x)$ is a delta function on the surface $\{y=\varphi_t(x)\}$, therefore
by~\cite[Theorem~8.2.4]{ho1} its wave
front set is contained in the conormal bundle to that surface:
$$
\WF(K)\subset \{(t,-V(x)\cdot\eta,\varphi_t(x),\eta,x,-{}^T\!d\varphi_t(x)\cdot\eta)\mid
t\in\mathbb R,\ x\in X,\ \eta\in T_{\varphi_t(x)}^*X\setminus 0\}.
$$
Then to prove~\eqref{e:lullaby}, we need to show that if
$t\neq 0$, $\varphi_t(x)=x$, $V(x)\cdot\eta=0$, and $(I-{}^T\!d\varphi_t(x))\cdot\eta=0$,
then $\eta=0$; this follows immediately from~\eqref{e:dyn-0}.

The principal component of the proof of the trace formula~\eqref{eq:guill} is the following
\begin{lemm}
  \label{l:guillemin-local}
Let $x_0\in X$ and $t_0\neq 0$ be such that $\varphi_{t_0}(x_0)=x_0$. Then there exists
$\varepsilon>0$ and a neighborhood $U\subset X$ of $x_0$ such that
$\varphi_s(x_0)\in U$ for $|s|<\varepsilon$ and
for each $\chi(t,x)\in C_{\mathrm c}^\infty((t_0-\varepsilon,t_0+\varepsilon)\times U)$,
we have
\begin{equation}
  \label{e:roo}
\int_{\mathbb R\times X} \chi(t,x)K(t,x,x)\,dx
={1\over |\det(I-\mathcal P_\gamma)|}\int_{-\varepsilon}^{\varepsilon}\chi(t_0,\varphi_s(x_0))\,ds,
\end{equation}
where $\mathcal P_\gamma$ is defined in~\eqref{e:p-gamma}.
\end{lemm}
\begin{proof}
We choose a local coordinate system $w=\psi(x)$,
$\psi:U_1\to B(0,\varepsilon_1)\subset\mathbb R^n$,
where $U_1$ is a neighborhood of $x_0$, such that
$$
\psi(x_0)=0,\quad
\psi_* V=\partial_{w_1},\quad
d\psi(x_0)\big(E_u(x_0)\oplus E_s(x_0)\big)=\{dw_1=0\}.
$$
We next choose small $\varepsilon\in (0,\varepsilon_1)$ such that for $U:=\psi^{-1}(B(0,\varepsilon))$
and $|t-t_0|<\varepsilon$, we have $\varphi_{-t}(U)\subset U_1$. We define the maps
$A:B_{\mathbb R^{n-1}}(0,\varepsilon)\to B_{\mathbb R^{n-1}}(0,\varepsilon_1)$ and
$F:B_{\mathbb R^{n-1}}(0,\varepsilon)\to (-\varepsilon_1,\varepsilon_1)$ by the formulas
$$
\varphi_{-t_0}(\psi^{-1}(0,w'))=\psi^{-1}(F(w'),A(w')),\quad
w'\in\mathbb R^{n-1},\
|w'|<\varepsilon.
$$
Then for $|t-t_0|<\varepsilon$ and $(w_1,w')\in B(0,\varepsilon)$, we have
$$
\varphi_{-t}(\psi^{-1}(w_1,w'))=\psi^{-1}(-t+t_0+w_1+F(w'),A(w')).
$$
Moreover, $F(0)=0$ and $A(0)=0$.

Since the flat trace does not depend on the choice of density on $X$,
we may choose the density $dx$ so that $\psi_* dx$ is the standard density
on $\mathbb R^n$. Then for $|t-t_0|<\varepsilon$ and $(z_1,z'),(w_1,w')\in B(0,\varepsilon)$,
we have
$$
K(t,\psi^{-1}(z_1,z'),\psi^{-1}(w_1,w'))=\delta(w'-A(z'))\delta(w_1+t-t_0-z_1-F(z')).
$$
The left-hand side of~\eqref{e:roo} is
$$
\int_{\mathbb R\times B(0,\varepsilon)} \chi(t,\psi^{-1}(w_1,w'))\delta(w'-A(w')) \delta(t-t_0-F(w'))\,dw_1dw'dt.
$$
Integrating out $t$, we get
$$
\int_{B(0,\varepsilon)} \chi(t_0+F(w'),\psi^{-1}(w_1,w'))\delta(w'-A(w'))\,dw_1dw'.
$$
Now, $dA(0)$ is conjugated by the map $d\psi(x_0)$ to the Poincar\'e map
$\mathcal P_\gamma$, therefore $I-dA(0)$ is invertible and for $\varepsilon$ small enough and
$|w'|<\varepsilon$, the equation $w'=A(w')$ has exactly one root at $w'=0$.
We then integrate out $w'$ to get
$$
{1\over |\det(I-d A( 0 ) )|}\int_{-\varepsilon}^\varepsilon\chi(t_0,\psi^{-1}(w_1,0))\,dw_1
={1\over |\det(I-\mathcal P_\gamma)|}\int_{-\varepsilon}^{\varepsilon}\chi(t_0,\varphi_s(x_0))\,ds,
$$
which finishes the proof.
\end{proof}

By Lemma~\ref{l:guillemin-local} and a partition of unity, we see that for each
$\chi(t,x)\in C_{\mathrm c}^\infty((\mathbb R\setminus 0)\times X)$, we have
\begin{equation}
  \label{e:guillemin-almost}
\int_{\mathbb R\times X}\chi(t,x)K(t,x,x)\,dx=\sum_{\gamma} {1\over |\det(I-\mathcal P_\gamma)|}\int_{\gamma}\chi(T_\gamma,x)\,dL(x)
\end{equation}
where the sum is over all closed trajectories $\gamma$ with period $T_\gamma$
and $dL$ refers to the measure $dt$ on $\gamma(t)=\varphi_t(x_0)$. By taking $\chi(t,x)=\chi(t)$,
we obtain~\eqref{eq:ABG}.

To show the more general~\eqref{eq:guill}, it suffices to prove a
local version similar to \eqref{e:roo}:
\begin{equation}
  \label{e:guillemin-there}
\int_{\mathbb R\times X}\chi(t,x)K^k(t,x,x)\,dx
={\tr(\wedge^k\mathcal P_\gamma)\over |\det(I-\mathcal P_\gamma)|}\int_{-\varepsilon}^\varepsilon \chi(t_0,\varphi_s(x_0))\,ds,
\end{equation}
where $K^k$ is the Schwartz kernel of the operator $\sum_{j=1}^r B_{jj}$,
$r=\dim \mathcal E_0^k$, and $B_{jl}:C_0^\infty(U)\to C^\infty(U)$ are the operators
defined by
$$
\varphi_{-t}^*(f\mathbf e_l)=\sum_{j=1}^r (B_{jl}(t)f)\mathbf e_j,
$$
here $\mathbf e_1,\dots,\mathbf e_r$ is a local frame of $\mathcal E_0^k$ defined near $x_0$.
Define the functions $b_{jl}$ on $(t_0-\varepsilon,t_0+\varepsilon)\times U$ by
$$
\varphi_{-t}^*\mathbf e_l=\sum_{j=1}^r b_{jl}(t)\mathbf e_j.
$$
Then $B_{jl}(t)f=b_{jl}(t)(\varphi_{-t}^*f)$, which means that
$$
K^k(t,x,y)=\sum_{j}b_{jj}(t,y) K(t,x,y),
$$
with $K(t,x,y)$ defined in~\eqref{e:defK}. Then by Lemma~\ref{l:guillemin-local},
$$
\int_{\mathbb R\times X}\chi(t,x)K^k(t,x,x)={1\over |\det(I-\mathcal P_\gamma)|}
\int_{-\varepsilon}^\varepsilon\chi(t_0,\varphi_s(x_0))\sum_j b_{jj}(t_0,\varphi_s(x_0))\,ds.
$$
It remains to note that
$$
\sum_j b_{jj}(t,\varphi_s(x_0))=
\tr\wedge^k ({}^T d\varphi_{-t_0}(x_0)|_{E_s^*(x_0)\oplus E_u^*(x_0)})=
\tr\wedge^k \mathcal P_\gamma.
$$

\section{Review of microlocal and semiclassical analysis}
\label{a:wf}

In this Appendix, we provide details and references for the concepts and facts
listed in~\S\ref{wfs}. All the proofs 
are essentially well known but we include them for the reader's
convenience. 

In standard microlocal analysis the asymptotic parameter is given by 
$ | \xi |$, where $ \xi $ is fiber variable (here the norm is
with respect to some smooth metric on the compact manifold  $ X$). 
We start our presentation with the review of that theory. In the 
semiclassical setting a small parameter $ h $ is added to 
measure the wave length of oscillations. We are 
then concerned in asymptotics as both $ h \to 0 $ and $ \xi \to 0 $.
That is one reason for which the fiber compactification is useful
as that provides a uniform setting for such 
asymptotics. In specific applications the
operators depend on additional parameters, in our case the 
spectral parameter $ \lambda$ or its rescaled version $ z = h \lambda $. 
If the classical objects (symbols) satisfy uniform estimates with 
respect to the parameters, so do their quantizations (operators),
as do the derivatives in $\lambda$.
That is implicit in many statements but is not stated in order
not to clutter the already complicated notation.

\subsection{Microlocal calculus}
  \label{a:wf-1}

Let $X$ be a manifold with a fixed volume form. We use the algebra
of pseudodifferential operators $\Psi^k(X)$, $k\in\mathbb R$, with symbols
lying in the class $S^k(X)\subset C^\infty(T^*X)$:
\begin{equation}
  \label{e:symbols}
a\in S^k(X)\ \Longleftrightarrow\ \sup_{x\in K}\langle\xi\rangle^{|\beta|-k}|\partial^\alpha_x\partial^\beta_{\xi}a(x,\xi)|\leq C_{\alpha\beta K},\quad
K\Subset X.
\end{equation}
See for example~\cite[\S18.1]{ho3} for the basic properties of operators in $\Psi^k$.
In particular, each $A\in \Psi^k(X)$ is bounded between Sobolev spaces
$H^m_{\comp}(X)\to H^{m-k}_{\loc}(X)$, or simply $H^m(X)\to H^{m-k}(X)$ if $X$ is compact.
The wave front set $\WF(A)$ of $A\in\Psi^k(X)$ is a closed conic subset of $T^*X\setminus 0$,
with $0$ denoting the zero section; the complement of $\WF(A)$
consists of points in whose conic neighbourhoods the full symbol of $A$ is
$\mathcal O(\langle\xi\rangle^{-\infty})$, see the discussion following~\cite[Proposition~18.1.26]{ho3}.

The wave front set $\WF(u)\subset T^*X\setminus 0$
of a distribution $u\in \mathcal D'(X)$ is defined
as follows: a point $(x,\xi)\in T^*X\setminus 0$
does not lie in $\WF(u)$ if there exists a conic neighbourhood $U$ of $(x,\xi)$ such that
$Au\in C^\infty(X)$ for each $A\in\Psi^0(X)$ with $\WF(A)\subset U$~-- see~\cite[(18.1.35) and Theorem~18.1.27]{ho3}.
An equivalent definition (see~\cite[Definition~8.1.2]{ho1}) is given in terms of the Fourier transform:
$(x,\xi)\not\in\WF(u)$ if and only if there exists $\chi\in C_{\rm{c}}^\infty(X)$ with
$\supp\chi$ contained in some coordinate neighbourhood and $\chi(x)\neq 0$
such that $\widehat{\chi u}(\xi')=\mathcal O(\langle\xi'\rangle^{-\infty})$
for $\xi'$ in a conic neighbourhood of $\xi$; here $\chi u$ is considered a function
on $\mathbb R^n$ using some coordinate system and $\xi$ is accordingly considered
as vector in $\mathbb R^n$.

The wave front set $\WF'(B)\subset T^*(Y\times X)$
of an operator $B:C_{\rm{c}}^\infty(X)\to \mathcal D'(Y)$ is defined using its
Schwartz kernel $K_B(y,x)\in \mathcal D'(Y\times X)$:
\begin{equation}
  \label{e:wf'}
\WF'(B):=\{(y,\eta,x,-\xi)\mid (y,\eta,x,\xi)\in\WF(K_B)\}.
\end{equation}
Here we use the fixed smooth density on $X$ to define the Schwartz kernel as
a distribution on $Y\times X$; however, this choice does not affect the
wave front set.
If $B\in\Psi^k(X)$, then the set defined in~\eqref{e:wf'} is the image
of the wave front set $\WF(B)\subset T^*X$ of $B$ as a pseudodifferential operator under the
diagonal embedding $T^*X\to T^*(X\times X)$, see~\cite[(18.1.34)]{ho3}.

The concept of the wave front plays a crucial role in the definition
of the flat trace. 
Before proving Lemma \ref{l:appr1} we give
\begin{proof}[Proof of \eqref{eq:eeps1}]
We first show that $ E_\epsilon \in \Psi^{0+} ( X ) $ with seminorm
estimates independent of $ \epsilon $.  For that we use Melrose's
characterization of pseudodifferential operators
\cite[\S18.2]{ho3}: it is enough to show that for any set of vector fields $ V_j
\in C^\infty ( X \times X; T ( X \times X ) )$ tangent to the
  diagonal, we have $ V_1 \cdots V_N K_{ E_\epsilon } \in H^{-n/2-} ( X \times
X ) $
with norm bounded uniformly in $\epsilon$.
This can be done in local coordinates, writing
$\psi(d(x,y)/\epsilon)=\Psi(x,(x-y)/\epsilon,\epsilon)$, where
$\Psi$ is a smooth function on $\mathbb R^n\times\RR^n\times [0,\infty)$,
compactly supported in the second argument.
We have $ F_\epsilon ( x ) = \int_{ \mathbb R^n } \Psi (x,(x-y)/\epsilon,\epsilon) J(y)\,dy$,
where $J$ is the Jacobian, and the support of the integrand lies $\mathcal O(\epsilon)$ close to $x$.
Then $  \partial_x^\alpha  F_\epsilon ( x )=
\mathcal O_\alpha ( \epsilon^n ) $; indeed, one can rewrite the $x$ derivatives
falling on the second argument of $\Psi$ as derivatives in $y$ and integrate by parts.
This implies that
$ \partial_x^\alpha  ( 1/ F_\epsilon ( x ) ) = \mathcal O_\alpha (
\epsilon^{-n} ) $. Locally, vector fields tangent to the diagonal
are generated by $ \partial_{x_j} + \partial_{y_j } $ and $ ( x_j -
y_j) \partial_{x_k} $ and we see that they preserve the class
of smooth functions of $x,(x-y)/\epsilon,\epsilon$. Therefore,
for $ 
|\alpha| = |\beta| $,
\[ ( x - y )^\alpha \partial_x^{\beta } ( \partial_x + \partial_y
)^\gamma K_{ E_\epsilon } ( x, y )  =  \epsilon^{-n} F_{\alpha \beta
  \gamma}  ( x, (x-y)/\epsilon, \epsilon ) , \]
where $ F_{\alpha \beta \gamma} \in C^\infty ( \RR^{2n}\times [0,\infty) ) $ are smooth functions.
The right hand side is in $ H^{-n/2-} ( \RR^{2n} ) $ uniformly in $\epsilon$ which
proves the claim.  To obtain\footnote{This specific statement is
not used in the paper: all we need is $ E_\epsilon \varphi \to \varphi
$ in $ C^\infty $ for $ \varphi \in C^\infty ( X) $, and that $ E_\epsilon $ is 
uniformly bounded in {\em some} $ \Psi^k ( X ) $.}
$ E_\epsilon \to I $ in  $ \Psi^{0+} ( X ) $ we apply the 
same argument to $ K_{E_\epsilon } - K_{I} $.
\end{proof}

\begin{proof}[Proof of Lemma~\ref{l:appr1}]
Let $ \Delta ( X ) = \{ ( x, x ) \} \subset X
\times X $ and let $ \Gamma $ be the complement of a small conic neighbourhood
of the conormal bundle $ N^*  \Delta ( X )  \subset T^* (X \times X ) $. 
Since $ \WF ( K_B  ) \cap N^*  \Delta ( X )  = \emptyset $ by~\eqref{e:no-diagonal} we
can choose $ \Gamma $ so that $ \WF ( K_B ) \subset \Gamma $.
This means that $ K_B \in {\mathcal  D}'_\Gamma ( X \times X ) $ where the last space consists of 
all distributions $ u \in {\mathcal D}' ( X \times X ) $ with  $
\WF ( u ) \subset \Gamma $.
If we write $ B_\epsilon := E_\epsilon B E_\epsilon $ then 
$ B_\epsilon : {\mathcal D}' ( X ) \to C^\infty ( X  ) $, and hence
$ K_{B_\epsilon } \in C^\infty ( X \times X ) $, 
\begin{equation}
  \label{e:get-trace-back}
\tr B_\epsilon = \int_X K_{B_\epsilon} (x, x ) dx = \int_X \iota^* K_{B_\epsilon}\,dx.
\end{equation}
Since $ E_\epsilon \to I $ in $ \Psi^{0+} $, $ E_\epsilon \varphi
\to \varphi $ in $ C^\infty  ( X )  $ for $ \varphi \in C^\infty ( X )
$. Hence  $ K_{B_\epsilon } ( \varphi_1 \otimes \varphi_2 )  \to K_B ( \varphi_1
\otimes \varphi_2 ) $, $ \varphi_j \in C^\infty ( X ) $, and 
consequently
$ K_{B_\epsilon}  \to K_B$ in $ {\mathcal D}' ( X \times X ) $. 
To show that  $ K_{B_\epsilon}  \to K_ B$ in $ {\mathcal D}'_\Gamma ( X
\times X ) $,  we adapt \cite[Definition 8.2.2]{ho1} and it
suffices to show that for each $ A \in \Psi^0 ( X \times X ) $
with $ \WF ( A ) \cap \Gamma = \emptyset$, $ A  K_{B_\epsilon}  $ is bounded in $C^\infty(X\times X)$ uniformly in $ \epsilon
$.
In fact,  
\[ A  K_{ B_\epsilon } =  A
E_{\epsilon,x}^t E_{\epsilon,y}  K_B  , \]
where $E_{\epsilon,x}$ and $E_{\epsilon,y}$ denote the operator
$E_\epsilon$ acting on $x$ and $y$ variables in $X\times X$, and the superscript
$t$ denotes the transpose.
Since $ E_\epsilon $ is uniformly bounded in 
$ \Psi^{0+} ( X ) $  and $ \WF ( A) $ is contained in a small 
neighbourhood of $ N^* \Delta ( X  ) $, 
$C_\epsilon  := A 
E_{\epsilon,x}^tE_{\epsilon,y}$
is in $ \Psi^{0+} ( X \times X ) $ with seminorms uniformly bounded with
respect to $ \epsilon $, and with $ \WF ( C_\epsilon ) \cap \Gamma =
\emptyset$.%
\footnote{The slight subtlety here lies in the fact that $
E_{\epsilon,x},E_{\epsilon,y}$ are {\em not } pseudifferential operators on $
X \times X $. However, the localization to a region where $ |\xi|  $ and $ |
\eta |$ are comparable makes the composition into a pseudodifferential operator.}
Hence $ C_\epsilon K_B \in C^\infty ( X \times X ) $ uniformly in $
\epsilon $ and thus $K_{B_\epsilon}\to K_B$ in $\mathcal D'_\Gamma(X\times X)$.
We now invoke \cite[Theorem 8.2.4]{ho1} to conclude that 
$ \iota^* K_{B_\epsilon } \to \iota^* K_B $ in $ \mathcal D' ( X ) $.
Hence $ \int_X \iota^* K_{B_\epsilon }\,dx \to \int_X \iota^* K_B\,dx $ as $ \epsilon\to 0 $, 
proving the lemma.
\end{proof}

If $\mathcal E$ is a smooth $r$-dimensional vector bundle over $X$ (see for example~\cite[Definition~6.4.2]{ho1}),
then we can consider distributions $\mathbf u\in \mathcal D'(X;\mathcal E)$ with values
in $\mathcal E$. The wave front set $\WF(\mathbf u)$, a closed conic subset
of $T^*X\setminus 0$, is defined as follows: $(x,\xi)\not\in\WF(\mathbf u)$ if and only
if for each local basis $\mathbf e_1,\dots,\mathbf e_r\in C^\infty(U;\mathcal E)$ of $\mathcal E$
defined in a neighbourhood $U$ of $x$, and for $\mathbf u|_U=\sum_{j=1}^r u_j\mathbf e_j$,
$u_j\in\mathcal D'(U)$, we have $(x,\xi)\not\in\WF(u_j)$ for all $j$. Similarly,
one can define $\WF'(\mathbf B)$ for an operator $\mathbf B$ with values in some smooth vector bundle
over $Y\times X$.

An operator $\mathbf A:\mathcal D'(X;\mathcal E)\to \mathcal D'(X;\mathcal E)$ is said to be pseudodifferential in the class
$\Psi^k(X)$, denoted $\mathbf A\in\Psi^k(X;\Hom(\mathcal E
))$, if
$\WF(\mathbf A\mathbf u)\subset \WF(\mathbf u)$ for all $u\in\mathcal D'(X;\mathcal E)$
and, for each local basis $\mathbf e_1,\dots,\mathbf e_r\in C^\infty(U;\mathcal E)$
over some open $U\subset X$, we have on $U$,
$$
\mathbf A (f\mathbf e_l)=\sum_{j=1}^r (A_{jl}f)\mathbf e_j,\quad\text{for each }
f\in\mathcal D'(X;\mathcal E),\
\supp f\Subset U,
$$
where $A_{jk}\in\Psi^k(U)$. As before, the wave front set $\WF(\mathbf A)$ on $U$
is defined as the union of $\WF(A_{jl})$ over all $j,l$. The principal symbol
$$
\sigma(\mathbf A)\in S^k(X;\Hom(\mathcal E 
))/S^{k-1}(X;\Hom(\mathcal E
))
$$
is defined using the standard notion of the principal symbol $\sigma(A_{jl})\in S^k(X)/S^{k-1}(X)$
(see the discussion following~\cite[Definition~18.1.20]{ho3}) as follows:
$$
\sigma(\mathbf A)\mathbf e_l=\sum_{j=1}^r \sigma(A_{jl})\mathbf e_j\quad\text{on }U.
$$
The operator $\mathbf A$ is called \emph{elliptic} in the class $\Psi^k$
at some point $(x,\xi)\in T^*X\setminus 0$, if $\langle\xi'\rangle^{-k}\sigma(\mathbf A)(x',\xi')$ is invertible
(as a homomorphism $\mathcal E\to\mathcal E$) uniformly as $\xi'\to\infty$ for $(x',\xi')$
in a conic neighbourhood of $(x,\xi)$; equivalently, $|\det(\langle\xi'\rangle^{-k}\sigma(\mathbf A))|\geq c>0$
in a conic neighbourhood of $(x,\xi)$. The (open conic) set of all elliptic points of $\mathbf A$
is denoted $\Ell(\mathbf A)$.


\subsection{Semiclassical calculus}
  \label{a:wf-2}

We now introduce the algebra $\Psi^k_h(X)$ of \emph{semiclassical} pseudodifferential
operators, depending on a parameter $h>0$ tending to zero~\cite[\S14.2]{ev-zw}.
The corresponding symbols $a(x,\xi;h)$ (denoted $a\in S^k_h(X)$)
satisfy $a(\cdot,\cdot;h)\in S^k(X)$ uniformly in
$h$ as $h\to 0$, with the class $S^k$ defined in~\eqref{e:symbols}.
Each $A\in\Psi^k_h(X)$ has a semiclassical wave front set
$\WFh(A)$, a closed (and not necessarily conic) subset of the
fiber-radially compactified cotangent bundle $\overline T^*X$
(see~\cite[\S2.1]{vasy1});
a point $(x,\xi)\in \overline T^*X$ does not lie in $\WFh(A)$ if and only if
the full symbol $a$ of $A$ satisfies $a(x',\xi')=\mathcal O(h^\infty\langle\xi'\rangle^{-\infty})$
for $h$ small enough and $(x',\xi')\in T^*X$ in a neighbourhood of $(x,\xi)$ in $\overline T^*X$.
The elements of $\Psi^k_h(X)$ act between semiclassical Sobolev
spaces $H^m_{h,\comp}(X)\to H^{m-k}_{h,\loc}(X)$ with norm $\mathcal O(1)$,
see~\cite[\S14.2.4]{ev-zw}.

Using operators in $\Psi^k_h(X)$, we define the semiclassical wave front set
$\WFh(u)\subset\overline T^*X$ for an $h$-tempered family of
distributions $u=u(h)$, see for example~\cite[\S8.4.2]{ev-zw},
\cite[\S3.1]{fwl}.
Similarly to $\WF(u)$, the set $\WFh(u)$ can be characterized using the Fourier transform as follows: $(x,\xi)\not\in\WFh(u)$
if and only if there exists $\chi\in C_{\rm{c}}^\infty(X)$ supported in some coordinate neighbourhood,
with $\chi(x)\neq 0$, and a neighbourhood $U_\xi$ of $\xi$ in $\overline T^*X$, such that
$ {\mathcal F}_h ( \chi u ) ( \xi') := \widehat{\chi u}(\xi'/h)=\mathcal O(h^\infty\langle\xi'\rangle^{-\infty})$
for $\xi'\in U_\xi$. This characterization immediately implies~\eqref{e:wf-wf-h}.
Similarly, one can define the wave front set $\WFh'(B)\subset\overline T^*(Y\times X)$
of an $h$-tempered family of operators $B(h):C_{\rm{c}}^\infty(X)\to \mathcal D'(Y)$.

The semiclassical principal symbol of $A\in \Psi^k_h(X)$, denoted
$\sigma_h(A)$, lies in the space $S^k_h(X)/hS^{k-1}_h(X)$~--
see~\cite[Theorem~14.1]{ev-zw}.
Note that this encodes the behaviour of the full symbol of $A$ at $h=0$ everywhere on $\overline T^*X$,
as well as the behaviour at the fiber infinity $\partial\overline T^*X$ for small, but positive, values
of $h$~-- see~\cite[\S2.1]{vasy1}. We cannot use the more convenient space of classical operators,
whose principal symbol is just a function on $T^*X$ (see~\cite[\S3.1]{fwl}) because the symbol of the operator
$ e^{ s G(h) } \mathbf P e^{-s G(h) } $ (see \S\ref{micro}) 
has the form $p+ishH_pG$, with $p\in S^1(X)$ and $H_pG=\mathcal O(\log(2+|\xi|))$ narrowly
missing the class $S^0(X)$. The (open) elliptic set $\Ell_h(A)\subset \overline T^*X$ is defined
as follows: $(x,\xi)\in\Ell_h(A)$ if $\langle\xi'\rangle^{-k}|\sigma_h(A)(x',\xi';h)|\geq c>0$
for $h$ small enough and all $(x',\xi')\in T^*X$ in a neighbourhood of $(x,\xi)$ in $\overline T^*X$.
Similarly to~\S\ref{a:wf-1}, we can study operators and distributions with values in smooth vector
bundles over $X$.

\begin{proof}[Proof of Lemma \ref{l:wfs}]
Using local coordinates, we reduce to the case $X=\mathbb R^n$, $Y=\mathbb R^m$.
Assume first that there exist neighbourhoods $U,V$ such that~\eqref{e:wf-char-op}
holds. Take $\chi_x\in C_{\rm{c}}^\infty(X),\chi_y\in C_{\rm{c}}^\infty(Y)$ with
$\chi_x(x)\neq 0,\chi_y(y)\neq 0$, and neighbourhoods $U_\xi,V_\eta$ of $\xi,\eta$,
such that
$
\supp\chi_x\times U_\xi\subset U,\quad \supp\chi_y\times V_\eta\subset V
$.

Let $K'_B(y',x')=\chi_y(y')K_B(y',x')\chi_x(x')$, and take arbitrary $\xi'\in U_\xi,\eta'\in V_\eta$
(depending on $h$). Then
$$
\mathcal F_h  {K'_B}(\eta',-\xi')= \mathcal F_h (\chi_yBf)(\eta'),\quad
f(x'):=\chi_x(x')e^{ix'\cdot\xi'/h}.
$$
where $ {\mathcal F_h } $ denotes the semiclassical Fourier transform
\cite[\S3.3]{ev-zw}.
We have $\WFh(f)\subset U$ (see~\cite[(8.4.7)]{ev-zw}) and thus by~\eqref{e:wf-char-op},
$\WFh(Bf)\cap V=\emptyset$. It follows that $\WFh(\chi_y Bf)\cap (\RR^n
\times V_\eta ) =\emptyset$
and thus by the semiclassical analog of~\cite[Proposition~8.1.3]{ho1},
 $ \mathcal F_h (\chi_yBf)(\eta')=\mathcal O(h^\infty)$ for $\eta'\in V_\eta$,
yielding, by the characterization of $\WFh$ via the Fourier transform, $(y,\eta,x,\xi)\not\in\WFh'(B)$.

Now, assume that $(y,\eta,x,\xi)\not\in\WFh'(B)$. Take $\chi_x\in C_{\rm{c}}^\infty(X),\chi_y\in C_{\rm{c}}^\infty(Y)$
such that $\chi_x=1$ on a neighbourhood $U_x$ of $x$,
$\chi_y=1$ on a neighbourhood $V_y$ of $y$, and neighbourhoods $U_\xi,V_\eta$ of $\xi,\eta$, such that
\begin{equation}
  \label{e:charles}
(\supp\chi_y\times \overline V_\eta\times\supp\chi_x\times \overline U_\xi)\cap\WFh'(B)=\emptyset.
\end{equation}
Put $U:=U_x\times U_\xi$, $ V:=V_y\times V_\eta$,
and assume that $f$ is an $h$-tempered family of distributions on $X$ such that
$\WFh(f)\subset U$. By Fourier inversion formula together with the characterization of $ \WFh $ via the Fourier transform,
\[ \begin{split}
f(x') & =\chi_x(x')(2\pi h)^{-n}\int_{U_\xi} e^{ix'\cdot\xi'/h} {\mathcal F}_h
f(\xi' )\,d\xi'+ ( 1 - \chi_x ( x')  ) f( x' ) \\
& \ \ \ \ \ \ \ \ \ 
 + \chi_x ( x' ) ( 2\pi h )^{-n} \int_{ \RR^n \setminus U_{\xi } }  e^{ix'\cdot\xi'/h} {\mathcal F}_h
f(\xi' )\,d\xi'  \\
& = (2\pi h)^{-n}\int_{U_\xi} \chi_x(x')e^{ix'\cdot\xi'/h} {\mathcal F}_h
f(\xi' )\,d\xi  + {\mathcal O} ( h^\infty )_{ C_{\rm{c}}^\infty } . 
\end{split}
\]
Therefore, if $K'_B(y',x')=\chi_y(y')K_B(y',x')\chi_x(x')$, then for bounded $\eta'$,
$$
\mathcal F_h (\chi_y Bf)(\eta')=(2\pi h)^{-n}\int_{U_\xi} \mathcal F_h
{K'_B}(\eta',-\xi')
\mathcal F_h  f(\xi')\,d\xi'+\mathcal O(h^\infty)_{\mathscr S(\mathbb R^m)}.
$$
However, we have by~\eqref{e:charles},
$ \mathcal F_h  {K'_B}(\eta',-\xi')=\mathcal O(h^\infty)$
for $(\eta',\xi')\in V_\eta\times U_\xi$; therefore,
$ \mathcal F_h (\chi_y Bf)(\eta')=\mathcal O(h^\infty)$ for $\eta'\in V_\eta$,
implying that $\WFh(Bf)\cap V=\emptyset.$
\end{proof}

\subsection{Proofs of semiclassical estimates}
  \label{a:wf-3}

In this subsection, we denote by boldface letters distributions with
values in $\mathcal E$ or operators acting on such distributions, and
with regular letters, scalar distributions and operators. Note that
any $A\in\Psi^k_h(X)$ can be viewed as an element of
$\Psi^k_h(X;\Hom(\mathcal E
))$ via the diagonal action.
\begin{proof}[Proof of Proposition~\ref{l:elliptic}]
Part~2 follows immediately from part~1 and the definition of $\WFh$. Indeed,
assume that $(x,\xi)\in \Ell_h(\mathbf P)\setminus\WFh(\mathbf P\mathbf u)$; it suffices
to prove that $(x,\xi)\not\in\WFh(\mathbf u)$. Take a neighbourhood $U$ of $(x,\xi)$
such that $U\Subset \Ell_h(\mathbf P)\setminus\WFh(\mathbf P\mathbf u)$,
and choose $B\in\Psi^0_h(X)$ such that
$U\subset\Ell_h(B)$ and $\WFh(B)\cap \WFh(\mathbf P\mathbf u)=\emptyset$.
Then $B\mathbf P$ is elliptic on $U$ and
$\|B\mathbf P\mathbf u\|_{H^{m-k}_h}=\mathcal O(h^\infty)$ for all $m$; by part~1, applied to
the operator $B\mathbf P$ in place of $\mathbf P$, we get
$\|A\mathbf u\|_{H^m_h}=\mathcal O(h^\infty)$ for all $m$ and
all $A\in\Psi^0_h(X)$ such that $\WFh(A)\subset U$, as required.

It remains to prove part~1.
Similarly to the proof of~\cite[Theorem~18.1.9]{ho3} (reducing to local frames of $\mathcal E$
and either using Cramer's rule or repeatedly differentiating the
equation $\sigma_h(\mathbf P)^{-1}\sigma_h(\mathbf P)=1$), we see that
the inverse $\sigma_h(\mathbf P)^{-1}$ of $\sigma_h(\mathbf P)$ in
$C^\infty(X;\Hom(\mathcal E
))$ is well-defined and lies in $S^{-k}_h(X;\Hom(\mathcal E
))$
for $h$ small enough and $(x,\xi)\in \Ell(\mathbf P)$. Using a cutoff function in $\overline T^*X$,
we can then construct $\mathbf q\in S^{-k}_h(X;\Hom(\mathcal
E
))$ such that
$\mathbf q=\sigma_h(\mathbf P)^{-1}$ near $\WFh(A)$. Take $\mathbf
Q_0\in\Psi^{-k}_h(X;\Hom(\mathcal E
))$
such that $\sigma_h(\mathbf Q_0)=\mathbf q$, then $\mathbf Q_0\mathbf P=1-h\mathbf R$ microlocally
near $\WFh(A)$, where $\mathbf R\in \Psi^{-1}_h(X;\Hom(\mathcal
E
))$. Using
asymptotic Neumann series exactly as in the proof of~\cite[Theorem~18.1.9]{ho3} to invert $1-h\mathbf R$, we construct
$\mathbf Q\in\Psi^{-k}_h(X;\Hom(\mathcal E
))$ such that
$$
\mathbf Q\mathbf P=1+\mathcal O(h^\infty)_{\Psi^{-\infty}}\quad\text{microlocally near }
\WFh(A).
$$
Then $A\mathbf u=A\mathbf Q\mathbf P\mathbf u+\mathcal O(h^\infty)_{C^\infty}$,
implying~\eqref{e:elliptic-est}.
\end{proof}
%
\begin{proof}[Proof of Proposition~\ref{l:hyperbolic}]
Similarly to Proposition~\ref{l:elliptic}, it is enough to prove part~1. Moreover,
by a partition of unity, we may assume that $\WFh(A)$ is contained
in a small neighbourhood of some fixed $(x_0,\xi_0)\in\overline T^*X$.
Let $\gamma(t)=\exp(tH_p)(x_0,\xi_0)$ and take $T\geq 0$ such that
$\gamma(-T)\in\Ell_h(B)$; we may then assume that
\begin{equation}
  \label{e:escape-prelude}
e^{-TH_p}(\WFh(A))\subset\Ell_h(B),\quad
e^{tH_p}(\WFh(A))\subset \Ell_h(B_1)\quad\text{for }t\in [-T,0].
\end{equation}
It is enough to prove the estimate
\begin{equation}
  \label{e:hyperbolic-int}
\|A\mathbf u\|_{H^m_h}\leq C\|B\mathbf u\|_{H^m_h}+Ch^{-1}\|B_1\mathbf P\mathbf u\|_{H^m_h}+\mathcal O(h^{1/2})\|B_1\mathbf u\|_{H_h^{m-1/2}}
+\mathcal O(h^\infty).
\end{equation}
Indeed, without loss of generality we may assume that each for each $(x,\xi)\in\WFh(B_1)$,
there exists $t\in [-T,0]$ such that $e^{tH_p}(x,\xi)\in \WFh(B)$;
one can then apply~\eqref{e:hyperbolic-int} with $A$ replaced by $B_1$ and
replace $\mathcal O(h^{1/2})\|B_1\mathbf u\|_{H_h^{m-1/2}}$ by $\mathcal O(h)\|B_2\mathbf u\|_{H_h^{m-1}}$
for certain $B_2\in\Psi^0_h$ microlocalized near $\gamma([-T,0])$; repeating this process,
and recalling that $\mathbf u$ is $h$-tempered,
we can ultimately make this term $\mathcal O(h^\infty)$.

In addition to a smooth density on $X$, we fix a smooth inner product on
the fibers of $\mathcal E$; this defines a Hilbert inner product $\langle\cdot,\cdot\rangle$
on $L^2(X;\mathcal E)$. We denote
$$
\Re \mathbf P={\mathbf P+\mathbf P^*\over 2},\quad
\Im \mathbf P={\mathbf P-\mathbf P^*\over 2i},
$$
so that $\Re\mathbf P,\Im\mathbf P\in\Psi^1_h(X;\Hom(\mathcal
E
))$ are symmetric and
$\mathbf P=\Re\mathbf P+i\Im\mathbf P$.


We will use an \emph{escape function}
$f(x,\xi)\in C^\infty(\overline T^*X)$, such that
$\supp f\subset\Ell_h(B_1)$ and
\begin{align}
\label{e:escape-1}
f\geq 0&\quad\text{everywhere};\\
\label{e:escape-2}
f>0&\quad\text{near }\WFh(A);\\
\label{e:escape-3}
H_p f\leq -C_0f&\quad\text{outside of }\Ell_h(B).
\end{align}
Here $C_0>0$ is a large constant to be chosen later.
To construct such $f$, we use~\eqref{e:escape-prelude} and identify
a tubular neighbourhood of $\gamma([-T,0])$ contained in $\Ell_h(B_1)$ with
$$
\{|\theta|<\delta\}\times (-T-\delta,\delta)_\tau\subset \mathbb R^{2n-1}_\theta\times\mathbb R_\tau,
$$
for small $\delta>0$, so that $H_p$ is mapped to $\partial_\tau$. We then put
$f(\theta,\tau)=\chi(\theta)\psi(\tau)$, where $\chi\in C_{\mathrm{c}}^\infty(\{|\theta|<\delta\};[0,1])$
satisfies $\chi=1$ on $\{|\theta|\leq\delta/2\}$, and
$\psi\in C_{\mathrm{c}}^\infty(-T-\delta,\delta)$ satisfies $\psi\geq 0$ everywhere,
$\psi(0)>0$, and
$\psi'\leq -C_0\psi$ outside of $(-T-\delta,-T+\delta)$. 
(To construct $ \psi $ we first choose $ \psi_0 \in
C_{\mathrm{c}}^\infty(-T-\delta,\delta)$ such that $ \psi_0 \geq 0 $,
$ \psi_0 ( 0 ) = 1 $, and $ \psi' \leq 0 $ on $ ( - T + \delta ,
\delta ) $.  We then put $ \psi ( \tau ) := e^{ - C_0 \tau } \psi_0
(\tau ) $.)


We now prove~\eqref{e:hyperbolic-int} by a positive commutator argument, going back to~\cite{ho-pc}.
Because $\WFh(A)$ might intersect the fiber infinity $\partial\overline T^*X$,
we have to put in regularizing pseudodifferential operators.
Assume that $S_\epsilon\in\Psi^{m-1}_h$, $\epsilon\in (0,1)$,
quantizes the symbol
$
\sigma_h(S_\epsilon) :=\langle\xi\rangle^m\langle\epsilon\xi\rangle^{-1}.
$
Note that $S_\epsilon$ is bounded uniformly in
$\Psi^{m}_h$ for $ \epsilon > 0 $. Take $F\in\Psi^0_h$ such that $\sigma_h(F)=f$
and $\WFh(F)\subset \Ell_h(B_1)$, and put $F_\epsilon=S_\epsilon F\in\Psi^{m-1}_h$, so that
$
\sigma_h(F_\epsilon)=f_\epsilon:=\langle\xi\rangle^m\langle\epsilon\xi\rangle^{-1}f
$.
Assume that $B_1\mathbf u\in H^{m-1/2}_h(X;\mathcal E)$. For each $\epsilon>0$
\begin{equation}
  \label{e:hyperbolic-commutator}
\Im\langle\mathbf P\mathbf u,F^*_\epsilon F_\epsilon\mathbf u\rangle
={i\over 2}\langle [\Re \mathbf P,F_\epsilon^*F_\epsilon]\mathbf u,\mathbf u\rangle
+{1\over 2}\langle (F_\epsilon^*F_\epsilon\Im\mathbf P+(\Im\mathbf P)F_\epsilon^*F_\epsilon)\mathbf u,\mathbf u\rangle,
\end{equation}
where the product on the left-hand side makes sense because
$B_1\mathbf P\mathbf u\in H^m_h \subset H^{m-3/2}_h$,
$\WFh(F_\epsilon)\subset\Ell_h(B_1)$ and $F^*_\epsilon F_\epsilon\mathbf u\in H^{-m+3/2}_h$.

We now estimate the terms on the right-hand side of~\eqref{e:hyperbolic-commutator}.
Denote
\begin{equation}
\label{eq:Tep}
\mathbf T_\epsilon:= {i\over 2h}[\Re\mathbf
P,F_\epsilon^*F_\epsilon]\in \Psi_h^{2m-2}(X;\Hom(\mathcal E
)),
\end{equation}
which is bounded in $\Psi_h^{2m}$, uniformly in $\epsilon$.
The principal symbol of $\mathbf T_\epsilon$ in $\Psi^{2m}_h$ is independent
of $h$ and diagonal with entries
\begin{equation}
  \label{e:f-eps-der}
f_\epsilon \,H_p f_\epsilon=
\langle\xi\rangle^m\langle\epsilon\xi\rangle^{-1}f_\epsilon\, H_pf
+f_\epsilon^2\bigg({m\over 2} \langle\xi\rangle^{-2}
-{\epsilon^2\over 2}\langle\epsilon\xi\rangle^{-2}\bigg)H_p(|\xi|^2).
\end{equation}
Since $H_p(|\xi|^2)=\mathcal O(|\xi|^2)$, we get
$$
\bigg({m\over 2} \langle\xi\rangle^{-2}
-{\epsilon^2\over 2}\langle\epsilon\xi\rangle^{-2}\bigg)H_p(|\xi|^2)=\mathcal O(1),
$$
uniformly in $\epsilon,\xi$. Therefore, for $C_0$ large enough depending on $m$,
and some large constant $C$, \eqref{e:escape-3} implies that
$$
f_\epsilon\,H_pf_\epsilon+{C_0\over 2} f_\epsilon^2\leq C|\langle\xi\rangle^m\sigma_h(B)|^2.
$$
The sharp G\r arding inequality~\cite[Theorem~9.11]{ev-zw} applied
to the operator
$\mathbf T_\epsilon+{C_0\over 2}F_\epsilon^*F_\epsilon-C(S_0B)^*(S_0B)$,
where $\sigma_h(S_0)=\langle\xi\rangle^{m}$,
gives, uniformly in $\epsilon$,
\begin{equation}
  \label{e:hyperbolic-commutator-1}
\langle\mathbf T_\epsilon\mathbf u,\mathbf u\rangle+{C_0\over 2}
\|F_\epsilon\mathbf u\|_{L^2}^2\leq C\|B\mathbf u\|_{H_h^m}^2+
Ch\|B_1\mathbf u\|_{H^{m-1/2}_h}^2+\mathcal O(h^\infty).
\end{equation}
We next claim that, uniformly in $\epsilon$,
\begin{equation}
  \label{e:hyperbolic-commutator-2}
{1\over 2}\langle (F_\epsilon^*F_\epsilon\Im \mathbf P+(\Im \mathbf P)F_\epsilon^*F_\epsilon)\mathbf u,\mathbf u\rangle
\leq C_1h\|F_\epsilon \mathbf u\|_{L^2}^2+Ch^2\|B_1\mathbf u\|_{H^{m-1/2}_h}^2+\mathcal O(h^\infty),
\end{equation}
where $C_1$ is a constant independent of the choice of $f$.
Indeed, the left-hand side of~\eqref{e:hyperbolic-commutator-2} can be written
as
$$
\langle (\Im \mathbf P)F_\epsilon \mathbf u,F_\epsilon\mathbf u\rangle
+{1\over 2}\langle (F_\epsilon^*[F_\epsilon,\Im\mathbf P]-[F_\epsilon^*,\Im\mathbf P]F_\epsilon)\mathbf u,
\mathbf u\rangle.
$$
Since $\sigma_h(\Im\mathbf P)=-q$ is diagonal
and nonpositive, the first term 
is bounded
from above by $C_1h\|F_\epsilon\mathbf u\|_{L^2}^2$ by the sharp G\r arding inequality.
The second term is bounded by $Ch^2\|B_1\mathbf u\|_{H^{m-1/2}_h}^2+\mathcal O(h^\infty)$, since
the principal symbol calculus shows that
$$F_\epsilon^*[F_\epsilon,\Im\mathbf P]-[F_\epsilon^*,\Im\mathbf
P]F_\epsilon\in h^2\Psi^{2m-1}_h $$
uniformly in $\epsilon$.

Combining~\eqref{e:hyperbolic-commutator}, \eqref{e:hyperbolic-commutator-1},
\eqref{e:hyperbolic-commutator-2}, taking $C_0>4C_1$, we get uniformly in $\epsilon$,
$$
{C_0\over 4}\|F_\epsilon\mathbf u\|_{L^2}^2\leq C\|B\mathbf u\|_{H^m_h}^2+
C h^{-1}\|B_1\mathbf P\mathbf u\|_{H^m_h}\|F_\epsilon\mathbf u\|_{L^2}
+Ch\|B_1\mathbf u\|_{H^{m-1/2}_h}^2+\mathcal O(h^\infty).
$$
Therefore, we have uniformly in $\epsilon$,
$$
\|F_\epsilon\mathbf u\|_{L^2}\leq C\|B\mathbf u\|_{H^m_h}+
Ch^{-1}\|B_1\mathbf P\mathbf u\|_{H^m_h}
+Ch^{1/2}\|B_1\mathbf u\|_{H^{m-1/2}_h}+\mathcal O(h^\infty).
$$
Now, $F_\epsilon=S_\epsilon F$ and
$S_\epsilon\to S_0$ in $\Psi^{m+1/2}_h$ as $\epsilon\to 0$;
therefore, $F_\epsilon \mathbf u\to S_0F\mathbf u$ in
$H^{-1}_h$.
Since $\|F_\epsilon\mathbf u\|_{L^2}$ is bounded uniformly in $\epsilon$,
by the compactness of the unit ball in $L^2$ in the weak topology
we get $S_0F\mathbf u\in L^2$; therefore, $F\mathbf u\in H^m_h$, and
$$
\|F\mathbf u\|_{H^m_h}\leq C\|B\mathbf u\|_{H^m_h}+
Ch^{-1}\|B_1\mathbf P\mathbf u\|_{H^m_h}
+Ch^{1/2}\|B_1\mathbf u\|_{H^{m-1/2}_h}+\mathcal O(h^\infty).
$$
It remains to apply the elliptic estimate~\eqref{e:elliptic-est}
together with~\eqref{e:escape-2}.
\end{proof}

To prove Propositions \ref{l:radial1} and \ref{l:radial2} we need the
following 

\begin{lemm}
\label{l:radialesc}
Suppose $ L $ is a radial source in the sense of  definition \eqref{eq:defL}. Then 
there exist:

1. $ f_0 \in C^\infty ( T^* X \setminus 0 ; [ 0 , 1 ] ) $,
  homogeneous of degree $ 0 $ and such that $ f_0 = 1 $ near $L$, $ \supp f_0
  \subset U $, and $ H_p f_0 \leq 0 $;

2.  $ f_1 \in C^\infty ( T^* X \setminus 0 ; [ 0 , \infty ) ) $, homogeneous of
  degree $ 1 $ and such that $f_1\geq c|\xi|$ everywhere and $ H_p f_1 \leq  -c f_1 $ on $U$, for some $
  c> 0 $.

\end{lemm}
\begin{proof}
To obtain part~1 we adapt the proof of \cite[Lemma 2.1]{fa-sj}.  Let $ V
= \kappa_* H_p $, where $ \kappa : T^* X\setminus 0 \to S^*X  \simeq (T^*X \setminus 0) / \RR_+ $ is the natural projection.
Since $ p $ is homogeneous of degree $ 1 $, $ \kappa_* H_p $ is a
smooth vector field on $ S^* X $, and the closed
set  $ \kappa ( L ) $  is invariant under the flow $ e^{ - t V } $. 
We will construct $ F \in C^\infty ( S^* X ; [ 0, 1 ]) $ 
such that $ V ( F ) \leq 0 $, $ \supp F \subset \kappa( U ) $ and $ F
= 1 $ on a neighbourhood of $ \kappa ( L ) $.   Then $ f_0 = \kappa^* F
$ will be a function satisfying the condition in part 1. 

To obtain $F$, fix $F_0\in C^\infty(S^*X;[0,1])$ such that $F_0=1$ near $\kappa(L)$
and $\supp F_0\subset \kappa(U)$. By the first assumption in~\eqref{eq:defL}, we have for $T>0$
large enough,
\begin{equation}
  \label{e:nina}
e^{-tV}\supp(F_0)\subset \{F_0=1\},\quad\text{for }t\geq T,
\end{equation}
and by the invariance of $ \kappa ( U ) $ by the flow, 
 $\supp (F_0\circ e^{tV})\subset \kappa(U)$ for all $t\geq T$.
Furthermore, $F_0(\rho)\geq F_0(e^{TV}(\rho))$ for all $\rho$;
indeed, if $e^{TV}(\rho)\in\supp F_0$, then $F_0(\rho)=1$
and otherwise $F_0(e^{TV}(\rho))=0$, and $0\leq F_0\leq 1$ everywhere.
Then the function
$$
F:={1\over T}\int_T^{2T}F_0\circ e^{tV}\,dt,\quad
V(F)={1\over T}(F_0\circ e^{2TV}-F_0\circ e^{TV}),
$$
satisfies the required assumptions.


The proof of part~2 is ``orthogonal'' to the proof of part~1 in the sense
that we are concerned about the radial component of $ H_p $. 
To find $ f_1 $, fix a smooth norm $|\cdot|$ of the fibers of $T^*X$.
By the second part of~\eqref{eq:defL}, we have
for $T_1$ large enough,
$$
|e^{-tH_p}(x,\xi)|\geq 2|\xi|,\quad\text{for }
(x,\xi)\in U,\
t\geq T_1.
$$
Then the function
$$
f_1(x,\xi):=\int_0^{T_1}|e^{-tH_p}(x,\xi)|\,dt,\quad
H_p f_1(x,\xi)=|\xi|-|e^{-T_1H_p}(x,\xi)|,
$$
is homogeneous of degree 1, $0<c|\xi|\leq f_1(x,\xi)\leq c^{-1}|\xi|$ everywhere,
and $H_pf_1(x,\xi)\leq -|\xi|\leq -cf_1(x,\xi)$ for $(x,\xi)\in U$.
\end{proof}

%
\begin{proof}[Proof of Proposition~\ref{l:radial1}]
As before, it is enough to prove part~1. Similarly to~\eqref{e:hyperbolic-int}, it suffices
to prove that for each $B_1\in\Psi^0_h$ elliptic on $\kappa(L)$, there exists
$A\in\Psi^0_h$ elliptic on $\kappa(L)$ such that for each $m\geq m_0$,
\begin{equation}
  \label{e:radial1-int}
\|A\mathbf u\|_{H^m_h}\leq Ch^{-1}\|B_1\mathbf P\mathbf u\|_{H^m_h}+\mathcal O(h^{1/2})\|B_1\mathbf u\|_{H^{m-1/2}_h}
+\mathcal O(h^\infty).
\end{equation}
Indeed, without loss of generality we may assume that $\WFh(B_1)\subset U$;
then by~\eqref{eq:defL}, each backward flow line of $H_p$ starting on $\WFh(B_1)$
reaches $\Ell_h(A)$. Combining~\eqref{e:radial1-int} with 
propagation of singularities (Proposition~\ref{l:hyperbolic}),
we see that for each $B'_1\in\Psi^0_h$ elliptic on $\kappa(L)$, there exists
$A\in\Psi^0_h$ elliptic on $\kappa(L)$ such that for each $m\geq m_0$,
$$
\|A\mathbf u\|_{H^m_h}\leq Ch^{-1}\|B'_1\mathbf P\mathbf u\|_{H^m_h}+\mathcal O(h^{1/2})
\|A\mathbf u\|_{H^{m-1/2}_h}+\mathcal O(h^\infty).
$$
Iterating this estimate, we arrive to
\begin{equation}
  \label{e:radial1-int-2}
\|A\mathbf u\|_{H^m_h}\leq Ch^{-1}\|B'_1\mathbf P\mathbf u\|_{H^m_h}+\mathcal O(h^\infty)
\|A\mathbf u\|_{H^{m_0}_h}+\mathcal O(h^\infty),
\end{equation}
and the $\mathcal O(h^\infty)\|A\mathbf u\|_{H^{m_0}_h}$
error term can be trivially removed provided that $A\mathbf u\in H^{m_0}_h$.

To prove~\eqref{e:radial1-int}, we shrink the conic neighbourhood $U$ of $L$ so that
$\kappa(U)\subset \Ell_h(B_1)$; here $\kappa:T^*X\setminus 0\to S^*X=\partial\overline T^*X$
is the natural projection to the fiber infinity.
Let $f_0,f_1$ be given by Lemma~\ref{l:radialesc} and consider $R>0$ large enough
so that $\supp f_0\cap \{f_1\geq R\}\subset \Ell_h(B_1)$.
Let $\chi\in C^\infty(\mathbb R;[0,1])$ satisfy $\supp\chi\subset (R,\infty)$,
$\chi=1$ on $[2R,\infty)$, and $\chi'\geq 0$ everywhere. Define $f\in C^\infty(\overline T^*X)$ by
\begin{equation}
  \label{e:radial-f}
f(x,\xi)=f_0(x,\xi)\chi(f_1(x,\xi)).
\end{equation}
It follows from Lemma~\ref{l:radialesc} that
$\supp f\subset \Ell_h(B_1)$, $f=1$ near $\kappa(L)$, and $H_p f\leq 0$ everywhere.

We now proceed as in the proof of Proposition~\ref{l:hyperbolic}, putting
$$
\sigma_h(S_\epsilon)=f_2^m\langle\epsilon\xi\rangle^{-1}.
$$
Here $f_2\in C^\infty(\overline T^*X)$ is positive everywhere and
is equal to $f_1$ for large $|\xi|$, in particular for $f_1(x, \xi)\geq R$.
If $f_\epsilon=\sigma_h(S_\epsilon)f$, then similarly to~\eqref{e:f-eps-der}, we find
\begin{equation}
  \label{e:f-eps-der1}
f_\epsilon H_p f_\epsilon=f_2^m\langle\epsilon\xi\rangle^{-1}f_\epsilon H_pf
+f_\epsilon^2\bigg(m{H_p f_2\over f_2}-{\epsilon^2H_p|\xi|^2\over 2\langle\epsilon\xi\rangle^2}\bigg)
\end{equation}
Since $H_pf\leq 0$ and $H_p f_2\leq -cf_2<0$ on $\supp f$, we see that for any fixed $C_0>0$,
$m_0$ large enough depending on $C_0$, and $m\geq m_0$, 
$$
f_\epsilon H_pf_\epsilon+ C_0 f_\epsilon^2\leq 0.
$$
Moreover, $m_0$ can be chosen independently of $B_1$.
For $ \mathbf T_\epsilon $ defined by \eqref{eq:Tep},  the sharp G\r
arding inequality gives, uniformly in $\epsilon$,
$$
\langle \mathbf T_\epsilon \mathbf u,\mathbf u\rangle+C_0\|F_\epsilon \mathbf u\|_{L^2}^2
\leq Ch \|B_1\mathbf u\|_{H^{m-1/2}_h}^2+\mathcal O(h^\infty).
$$
Arguing as in the proof of Proposition~\ref{l:hyperbolic}, we obtain~\eqref{e:radial1-int} with
$A:=F$.
\end{proof}
%
\begin{proof}[Proof of Proposition~\ref{l:radial2}]
We proceed as in the proof of Proposition~\ref{l:radial1}, showing that
for each $B_1\in\Psi^0_h$ elliptic on $\kappa(L)$, there exists
$A\in\Psi^0_h(X)$ elliptic on $\kappa(L)$ and
$B\in\Psi^0_h(X)$ with $\WFh(B)\subset \Ell_h(B_1)\setminus \kappa(L)$
such that for $m\leq -m_0$,
\begin{equation}
 \label{e:radial2-int}
\|A\mathbf u\|_{H^m_h}\leq C\|B\mathbf u\|_{H^m_h}+Ch^{-1}\|B_1\mathbf P\mathbf u\|_{H^m_h}
+\mathcal O(h^{1/2})\|B_1\mathbf u\|_{H^{m-1/2}_h}
+\mathcal O(h^\infty).
\end{equation}
Take $f\in C^\infty(\overline T^*X;[0,1])$ such that $\supp f\subset \Ell_h(B_1)$
and $f=1$ near $\kappa(L)$, and define $f_2$ using Lemma~\ref{l:radialesc}
with the sign of $p$ reversed, so that $H_p f_2\geq cf_2$ on $\supp f$.
We define $S_\epsilon,f_\epsilon$ as in the proof of Proposition~\ref{l:radial2}
and analyse the terms on the right-hand side of~\eqref{e:f-eps-der1}.
The first term
vanishes near $\kappa(L)$ since $f=1$ there. Using the second term, we see that
for each $C_0$, $m_0$ large enough depending on $C_0$, and $m\leq -m_0$,
$$
f_\epsilon H_p f_\epsilon+C_0f_\epsilon^2\leq |\langle\xi\rangle^{m}\sigma_h(B)|^2,
$$
for some choice of $B\in\Psi^0_h$ with $\WFh(B)\subset\Ell_h(B_1)\setminus\kappa(L)$.
By sharp G\r arding inequality, we have uniformly in $\epsilon$
$$
\langle\mathbf T_\epsilon\mathbf u,\mathbf u\rangle+C_0\|F_\epsilon\mathbf u\|_{L^2}^2
\leq C\|B\mathbf u\|_{H^m_h}^2+Ch\|B_1\mathbf u\|_{H^{m-1/2}_h}^2+\mathcal O(h^\infty);
$$
arguing as in the proof of Proposition~\ref{l:hyperbolic}, we obtain~\eqref{e:radial2-int} with
$A:=F$.
\end{proof}

\smallsection{Acknowledgements} 
We would like to thank Colin Guillarmou and Fr\'ederic Naud for
helpful discussions and in particular for pointing out that our result
holds under the condition~\eqref{eq:Poinc} and not just for contact
flows. We would also like to thank the anonymous referees
for corrections and valuable suggestions.
Partial support by the National Science Foundation under the grant
DMS-1201417 is also gratefully acknowledged.

%


\end{document}